\documentclass[11pt,a4paper]{amsart}
\usepackage{amssymb}
\usepackage[a4paper,centering,twoside,textwidth=6in]{geometry}

\newcommand{\De}{\Delta}
\newcommand{\al}{\alpha}
\newcommand{\be}{\beta}
\newcommand{\ga}{\gamma}

\newcommand{\la}{\lambda}

\DeclareMathOperator{\PSL}{PSL} \DeclareMathOperator{\PGL}{PGL}
\DeclareMathOperator{\PSU}{PSU} \DeclareMathOperator{\SL}{SL}
\DeclareMathOperator{\GL}{GL} 
 \DeclareMathOperator{\Aut}{Aut}
\DeclareMathOperator{\tr}{tr}\DeclareMathOperator{\lcm}{lcm}

\newcommand{\Tr}{\mathcal{T}}
\newcommand{\Pri}{\mathcal{P}}
\newcommand{\Dq}{\mathcal{D}}
\newcommand{\Sq}{\mathcal{S}}
\newcommand{\Nq}{\mathcal{N}}
\newcommand{\cC}{\mathcal{C}}

\theoremstyle{plain}
\newtheorem{prop}{Proposition}[section]
\newtheorem{lemma}[prop]{Lemma}
\newtheorem{corr}[prop]{Corollary}

\newtheorem{thm}[prop]{Theorem}
\newtheorem{theorem}{Theorem}

\theoremstyle{definition}

\newtheorem{defn}[prop]{Definition}

\newtheorem{example}[prop]{Example}
\newtheorem{rem}[prop]{Remark}

\title{On Beauville Structures for $\PSL_2(q)$}
\author{Shelly Garion}
\thanks{The author was supported by a European Postdoctoral Fellowship (EPDI)
and by the SFB 878 ``Groups, Geometry and Actions''.}

\address{Max-Planck-Institute for Mathematics\\ D-53111 Bonn, Germany}
\curraddr{Fachbereich Mathematik und Informatik, Universit\"at
M\"unster, Einsteinstrasse 62, D-48149 M\"unster, Germany}
\email{shelly.garion@uni-muenster.de}

\begin{document}

\maketitle

\begin{abstract}
We characterize Beauville surfaces of unmixed type with group either
$\PSL_2(p^e)$ or $\PGL_2(p^e)$, thus extending previous results of
Bauer, Catanese and Grunewald, Fuertes and Jones, and Penegini and
the author.
\end{abstract}

\section{Introduction}\label{sect.intro}


A \emph{Beauville surface} $S$ (over $\mathbb{C}$) is a particular
kind of surface isogenous to a \emph{higher product of curves},
i.e., $S=(C_1 \times C_2)/G$ is a quotient of a product of two
smooth curves $C_1$ and $C_2$ of genus at least two, modulo a free
action of a finite group $G$ which acts faithfully on each curve.
For Beauville surfaces the quotients $C_i/G$ are isomorphic to
$\mathbb{P}^1$ and both projections $C_i \rightarrow C_i/G \cong
\mathbb{P}^1$ are coverings branched over three points. Beauville
surfaces were introduced by Catanese in~\cite{Cat00}, inspired by a
construction of Beauville \cite{Be}.

A Beauville surface $S$ is either of \emph{mixed} or \emph{unmixed}
type according respectively as the action of $G$ exchanges the two
factors (and then $C_1$ and $C_2$ are isomorphic) or $G$ acts
diagonally on the product $C_1 \times C_2$. The subgroup $G_0$ (of
index $\leq 2$) of $G$ which preserves the ordered pair $(C_1,C_2)$
is then respectively of index $2$ or $1$ in $G$.

Any Beauville surface $S$ can be presented in such a way that the
subgroup $G_0$ of $G$ acts effectively on each of the factors $C_1$
and $C_2$. Catanese called such a presentation \emph{minimal} and
proved its uniqueness in \cite{Cat00}.

In this paper we shall consider only Beauville surfaces of unmixed
type so that $G_0 = G$. A natural question is to determine the
finite groups which characterize unmixed Beauville surfaces in a
minimal presentation. Since a finite group appears as the underlying
group of an unmixed Beauville surface in a minimal presentation if
and only if it admits an unmixed Beauville structure (see
~\cite{BCG05,BCG06}), the above question is equivalent to
determining the finite groups admitting an unmixed Beauville
structure.

\begin{defn}\label{defn.beau}
An \emph{unmixed Beauville structure} for a finite group $G$
consists of two triples $(a_1,b_1,c_1)$ and $(a_2,b_2,c_2)$ of
elements of $G$ which satisfy
\begin{enumerate}\renewcommand{\theenumi}{\it \roman{enumi}}
    \item $a_1b_1c_1=1$ and $a_2b_2c_2=1$,
    \item $\langle a_1,b_1 \rangle = G$ and $\langle a_2,b_2 \rangle=G$,
    \item $\Sigma(a_1,b_1,c_1) \cap \Sigma(a_2,b_2,c_2)= \{1\}$, \\
    where, for $i \in \{1,2\}$,
    $\Sigma(a_i,b_i,c_i)$ is the union of the conjugacy classes of
    all powers of $a_i$, all powers of $b_i$, and all powers of
    $c_i$.
\end{enumerate}
\end{defn}

Moreover denoting the order of an element $g$ in $G$ by $|g|$, we
define the \emph{type} $\tau_i$ of $(a_i,b_i,c_i)$ to be the triple
$(|a_i|,|b_i|,|c_i|)$. In this situation, we say that $G$ admits an
\emph{unmixed Beauville structure of type} $(\tau_1,\tau_2)$.

The question whether a finite group admits an unmixed Beauville
structure of a given type is closely related to the question whether
it is a quotient of certain triangle groups. More precisely, a
necessary condition for a finite group $G$ to admit an unmixed
Beauville structure of type $(\tau_1,\tau_2) =
\bigl((r_1,s_1,t_1),(r_2,s_2,t_2)\bigr)$ is that $G$ is a quotient
with torsion free-kernel of the triangle groups $T_{r_1,s_1,t_1}$
and $T_{r_2,s_2,t_2}$, where for $i\in \{1,2\}$,
\[
    T_{r_i,s_i,t_i} = \langle x,y,z: x^{r_i}=y^{s_i}=z^{t_i}=xyz=1 \rangle.
\]
Indeed, conditions \emph{(i)} and \emph{(ii)} of
Definition~\ref{defn.beau} are equivalent to the condition that $G$
is a quotient of each of the triangle groups
$T_{|a_i|,|b_i|,|c_i|}$, for $i\in \{1,2\}$, with torsion-free
kernel.

When investigating the existence of an unmixed Beauville structure
for a finite group, one can consider only types $(\tau_1,\tau_2)$,
where for $i\in \{1,2\}$, $\tau_i = (r_i,s_i,t_i)$ satisfies $1/r_i
+ 1/s_i + 1/t_i < 1$. Then $T_{r_i,s_i,t_i}$ is a (infinite
non-soluble) hyperbolic triangle group and we say that $\tau_i$ is
\emph{hyperbolic}.

Indeed, if $1/r_i + 1/s_i + 1/t_i > 1$ then $T_{r_i,s_i,t_i}$ is a
finite group, and moreover, it is either dihedral or isomorphic to
one of $A_4$, $A_5$ or $S_4$. By \cite[Proposition 3.6 and Lemma
3.7]{BCG05}, in these cases $G$ cannot admit an unmixed Beauville
structure. If $1/r_i + 1/s_i + 1/t_i = 1$ then $T_{r_i,s_i,t_i}$ is
one of the (soluble infinite) ``wall-paper'' groups, and by \cite[\S
6]{BCG05}, none of its finite quotients can admit an unmixed
Beauville structure.

A considerable effort has been made to classify the finite simple
groups which admit an unmixed Beauville structure. A finite abelian
simple group clearly does not admit an unmixed Beauville structure,
since by~\cite[Theorem 3.4]{BCG05} the only finite abelian groups
admitting an unmixed Beauville structure are the abelian groups of
the form $Z_n \times Z_n$ where $n$ is a positive integer coprime to
$6$. (Here $Z_n$ denotes a cyclic group of order $n$.) In
\cite{BCG05}, Bauer, Catanese and Grunewald provided the first
results on finite non-abelian simple groups admitting an unmixed
Beauville structure, and conjectured that all finite non-abelian
simple groups admit an unmixed Beauville structure with the
exception of $A_5$.

This conjecture has received much attention and has recently been
proved to hold. Concerning the simple alternating groups, it was
established in \cite{FG} that $A_5$ is indeed the only one not
admitting an unmixed Beauville structure. In \cite{FJ,GP}, the
conjecture is shown to hold for the projective special linear groups
$\PSL_2(q)$ (where $q > 5$), the Suzuki groups $^2B_2(q)$ and the
Ree groups $^2G_2(q)$ as well as some other families of finite
simple groups of Lie type of small rank (where $q$ is sufficiently
large). The next major result concerning the investigation of the
conjecture with respect to the finite simple groups of Lie type was
pursued by Garion, Larsen and Lubotzky who showed in \cite{GLL} that
the conjecture holds for finite non-abelian simple groups of
sufficiently large order. The final step regarding the investigation
of the conjecture was carried out by Guralnick and Malle \cite{GM}
and Fairbairn, Magaard and Parker \cite{FMP} who established its
veracity in general.

There has also been an effort to classify the finite quasisimple
groups and almost simple groups which admit an unmixed Beauville
structure. Recall that a finite group $G$ is \emph{quasisimple}
provided $G/Z(G)$ is a non-abelian simple group and $G = [G,G]$. In
\cite{FJ} it was shown that $\SL_2(q)$ (for $q>5$) admits an unmixed
Beauville structure. Fairbairn, Magaard and Parker \cite{FMP} showed
that with the exceptions of $\SL_2(5)$ and $\PSL_2(5) \cong \SL_2(4)
\cong A_5$, every finite quasisimple group admits an unmixed
Beauville structure. By \cite{BCG06,FG}, the almost simple symmetric
groups $S_n$ (where $n\geq 5$) admit an unmixed Beauville structure.
Recall that a group $G$ is called \emph{almost simple} if there is a
non-abelian simple group $G_0$ such that $G_0 \leq G \leq
\Aut(G_0)$.

Another conjecture of Bauer, Catanese and Grunewald proposed in
\cite{BCG05} states that if $\tau_1=(r_1,s_1,t_1)$ and
$\tau_2=(r_2,s_2,t_2)$ are two hyperbolic types, then almost all
alternating groups $A_n$ admit an unmixed Beauville structure of
type $(\tau_1,\tau_2)$. This has recently been proved in~\cite{GP},
based on results of Liebeck and Shalev~\cite{LS04}, where a similar
conjecture is raised, replacing $A_n$ by a finite simple classical
group of Lie type of sufficiently large Lie rank.

In contrast, when the Lie rank is very small, as in the case of
$\PSL_2(q)$, such a conjecture does not hold. It is therefore the
aim of this paper to characterize the possible types of an unmixed
Beauville structure for the projective special linear group
$\PSL_2(q)$. This is done in Theorem~\ref{thm.Beau.PSL2q}. A similar
result for the projective general linear group $\PGL_2(q)$ is
described in Theorem~\ref{thm.Beau.PGL2q}. In particular we show
that the almost simple group $\PGL_2(q)$ (where $q \geq 5$) admits
an unmixed Beauville structure.

\subsection*{Beauville structures for $\PSL_2(q)$ and $\PGL_2(q)$}
If $H = \PSL_2(q)$ (respectively, $\PGL_2(q)$) with $q\leq 5$
(respectively, $q\leq 4$) then $H$ is isomorphic to one of
$S_3,S_4,A_4$ or $A_5$. As none of these groups admits an unmixed
Beauville structure by \cite[Proposition 3.6]{BCG05}, we can assume
hereafter that $q\geq 7$ or $q\geq 5$ according respectively as $H =
\PSL_2(q)$ or $\PGL_2(q)$. Unless otherwise stated, we also let $G =
\PSL_2(q)$ and $G_1 = \PGL_2(q)$ where $q = p^e$ for some prime
number $p$ and some positive integer $e$.

Our first result is the characterization of the possible types of
unmixed Beauville structures for $\PSL_2(q)$.

\begin{theorem}\label{thm.Beau.PSL2q}
Let $G = \PSL_2(q)$ where $5 < q = p^e$ for some prime number $p$
and some positive integer $e$. Let $\tau_1 = (r_1,s_1,t_1)$, $\tau_2
= (r_2,s_2,t_2)$ be two hyperbolic triples of integers. Then $G$
admits an unmixed Beauville structure of type $(\tau_1,\tau_2)$ if
and only if the following hold:
\begin{enumerate}\renewcommand{\theenumi}{\it \roman{enumi}}
\item $G$ is a quotient of $T_{r_1,s_1,t_1}$ and $T_{r_2,s_2,t_2}$ with torsion-free
kernel. \\
Equivalently, $(e,\tau_1)$ and $(e,\tau_2)$ satisfy the conditions
given in Table~\ref{table.triple.PSL} in Section~\ref{sect.LR}.

\item If $p=2$ or $e$ is odd or $q=9$, then
$r_1s_1t_1$ is coprime to $r_2s_2t_2$. \\
If $p$ is odd, $e$ is even and $q>9$, then
$g=\gcd(r_1s_1t_1,r_2s_2t_2) \in \{1,p,p^2\}$. Moreover, if $p$
divides $g$ and $\tau_1$ (respectively $\tau_2)$ is up to a
permutation $(p,p,n)$ then $n \ne p$ and $n$ is a \emph{good
$G$-order} (see Definition~\ref{defn.good.G.order}).
\end{enumerate}
\end{theorem}

\begin{defn}\label{defn.good.G.order}
Let $q$ be an odd prime power, let $G=\PSL_2(q)$ and let $n>1$ be an
integer. Then $n$ is called a \emph{good $G$-order} if one of the
following holds:
\begin{itemize}
\item $n$ is odd and divides either $q-1$ or $q+1$;
\item $n$ is even and $4n$ divides either $q-1$ or $q+1$.
\end{itemize}
\end{defn}

We deduce from Theorem~\ref{thm.Beau.PSL2q} that for any $q>7$ the
group $\PSL_2(q)$ admits unmixed Beauville structures of types
\[
    \left(\bigl(\frac{q-1}{d},\frac{q-1}{d},\frac{q-1}{d}\bigr),
    \bigl(\frac{q+1}{d},\frac{q+1}{d},\frac{q+1}{d}\bigr)\right),
\]
and
\[
    \left(\bigl(\frac{q-1}{d},\frac{q-1}{d},\frac{q-1}{d}\bigr),
    \bigl(\frac{q+1}{d},\frac{q+1}{d},p\bigr)\right),
\]
where $d=\gcd(2,q-1)$, thus recovering the results appearing
in~\cite{GP} and~\cite{FJ} respectively (the case $q=7$ is excluded
since the triple $(3,3,3)$ is not hyperbolic).
In addition, if $q \geq 7$ and $q \neq 9$ then $\PSL_2(q)$
admits an unmixed Beauville structure of type
\[
    \left(\bigl(\frac{q-1}{d},\frac{q-1}{d},p\bigr),
    \bigl(\frac{q+1}{d},\frac{q+1}{d},\frac{q+1}{d}\bigr)\right),
\]
(the case $q=9$ is excluded since $\PSL_2(9)$ is not a quotient of
$T_{3,4,4}$, see Table~\ref{table.Ex.triples.5.7.8.9} in
Section~\ref{sect.LR}).

When $q\geq 7$ is odd, $\PSL_2(q)$ admits an unmixed Beauville
structure of type
\[
    \left(\bigl(p,p,\frac{q-1}{2}\bigr), \bigl(\frac{q+1}{2},\frac{q+1}{2},\frac{q+1}{2}\bigr)\right),
\]
and when $q>7$, it also admits an unmixed Beauville structure of
type
\[
    \left(\bigl(p,p,\frac{q+1}{2}\bigr), \bigl(\frac{q-1}{2},\frac{q-1}{2},\frac{q-1}{2}\bigr)\right),
\]
(again, the case $q=7$ is excluded since the triple $(3,3,3)$ is not hyperbolic).

If $p \geq 7$ is prime then $\PSL_2(p)$ admits unmixed Beauville
structures of types
\[
    \left(\bigl(p,p,p\bigr), \bigl(\frac{p-1}{2},\frac{p-1}{2},\frac{p+1}{2}\bigr)\right),
    \left(\bigl(p,p,p\bigr),\bigl(\frac{p-1}{2},\frac{p+1}{2},\frac{p+1}{2}\bigr)\right),
    \left(\bigl(p,p,p\bigr),\bigl(\frac{p+1}{2},\frac{p+1}{2},\frac{p+1}{2}\bigr)\right),
\]
and if $p>7$ is prime, $\PSL_2(p)$ also admits an unmixed Beauville structure of type
\[
    \left(\bigl(p,p,p\bigr), \bigl(\frac{p-1}{2},\frac{p-1}{2},\frac{p-1}{2}\bigr)\right).
\]
Note that $\PSL_2(p^e)$ is a quotient of $T_{p,p,p}$ if and only if $e=1$,
by Table~\ref{table.triple.PSL} in Section~\ref{sect.LR}.

\begin{example} We list below \emph{all} the possible types for Beauville
structures for $\PSL_2(q)$ where $q=7,8,9$. This is a direct consequence of
Theorem~\ref{thm.Beau.PSL2q}, and in particular it follows from Table~\ref{table.Ex.triples.5.7.8.9}
in Section~\ref{sect.LR}. This table describes all the hyperbolic triples satisfying condition~{\it(i)},
and one can construct all possible pairs of these triples and check whether they
are coprime, thus also satisfying condition~{\it(ii)}.
It can easily be verified by a computer, for example using \textsc{Magma}.
\begin{itemize}
\item $\PSL_2(7)$:
\[
    \bigl((3,7,7),(4,4,4)\bigr), \bigl((3, 4, 4),(7, 7, 7)\bigr),
    \bigl((3, 3, 7),(4, 4, 4)\bigr),
\]
\[
    \bigl((4, 4, 4),(7, 7, 7)\bigr), \bigl((3, 3, 4),(7, 7, 7)\bigr).
\]
(see also~\cite[Theorem 13]{FGJ}).
\item $\PSL_2(8)$:
\[
    \bigl((2, 7, 7),(9, 9, 9)\bigr), \bigl((2, 9, 9 ),(7, 7, 7)\bigr),
    \bigl((7, 7, 7 ),(9, 9, 9)\bigr), \bigl((2, 3, 9),(7, 7, 7)\bigr),
\]
\[
    \bigl((3, 9, 9),(7, 7, 7)\bigr), \bigl((3, 3, 9),(7, 7, 7)\bigr),
    \bigl((2, 7, 7),(3, 9, 9)\bigr), \bigl((2, 7, 7),(3, 3, 9)\bigr).
\]
\item $\PSL_2(9) \cong A_6$:
\[
    \bigl((3, 5, 5),(4, 4, 4)\bigr), \bigl((4, 4, 4),(5, 5, 5)\bigr),
    \bigl((3, 3, 5),(4, 4, 4)\bigr), \bigl((3, 3, 4),(5, 5, 5)\bigr).
\]
\end{itemize}
\end{example}

Observe that condition~\textit{(iii)} of Definition~\ref{defn.beau}
is clearly satisfied under the assumption that $r_1 s_1 t_1$ is
coprime to $r_2 s_2 t_2$. The example of the alternating groups
shows that this assumption is not always necessary. But in the case
of $\PSL_2(q)$ Theorem~\ref{thm.Beau.PSL2q} shows that this
assumption is actually not far from being necessary.

However, by Theorem~\ref{thm.Beau.PSL2q}, for any odd prime power
$q>3$ the group $\PSL_2(q^2)$ admits an unmixed
Beauville structure of type
\[
    \left(p,\frac{q^2-1}{2},\frac{q^2-1}{2}\bigr),
    \bigl(p,\frac{q^2+1}{2},\frac{q^2+1}{2}\bigr)\right).
\]
In particular, $\PSL_2(q^2)$ is a quotient of the hyperbolic
triangle groups $T_{p, (q^2 \pm 1)/2, (q^2 \pm 1)/2}$ by
Table~\ref{table.triple.PSL} in Section~\ref{sect.LR} and
Lemma~\ref{lem.irr}, since $\PSL_2(q^2)$ contains elements of orders
$(q^2 \pm 1)/2$, but none of its subfield subgroups contain such
elements. In addition, $\PSL_2(q^2)$ also admits unmixed Beauville
structures of types
$$\bigl((p,p,t_1),(p,p,t_2)\bigr)$$ for certain $t_1,t_2$ dividing
$(q^2-1)/2$, $(q^2+1)/2$ respectively (see Lemma~\ref{lem.ppt}).

\begin{example}
We list below \emph{all} the possible types of the form
$\bigl((p,p,t_1), (p,p,t_2)\bigr)$ for Beauville structures for
$\PSL_2(q^2)$ where $q=5,7,11,13$. This is a direct consequence of
Theorem~\ref{thm.Beau.PSL2q}, and in particular it follows from Table~\ref{table.good.G.orders}
in Section~\ref{sect.unip}.  This table describes all the \emph{good $G$-orders},
and so one needs only to check when they are pairwise coprime, thus satisfying condition~{\it(ii)}.
It can easily be verified by a computer, for example using \textsc{Magma}.
\begin{itemize}
\item $\PSL_2(25)$:
$$\bigl((5, 5, 6),(5, 5, 13)\bigr).$$
\item $\PSL_2(49)$:
\[
    \bigl((7, 7, 5),(7, 7, 6)\bigr), \bigl((7, 7, 5),(7, 7, 12)\bigr),
    \bigl((7, 7, 25),(7, 7, 6)\bigr), \bigl((7, 7, 25),(7, 7,
    12)\bigr).
\]
\item $\PSL_2(121)$:
\[
    \bigl((11,11,10),(11,11,61)\bigr), \bigl((11,11,15),(11,11,61)\bigr),
    \bigl((11,11,30),(11,11,61)\bigr).
\]
\item $\PSL_2(169)$:
\[
    \bigl((13,13,5),(13,13,14)\bigr), \bigl((13,13,5),(13,13,17)\bigr),
    \bigl((13,13,5),(13,13,21)\bigr),
\]
\[
    \bigl((13,13,5),(13,13,42)\bigr), \bigl((13,13,14),(13,13,17)\bigr),
    \bigl((13,13,14),(13,13,85)\bigr),
\]
\[
    \bigl((13,13,17),(13,13,21)\bigr), \bigl((13,13,17),(13,13,42)\bigr),
    \bigl((13,13,21),(13,13,85)\bigr),
\]
\[
    \bigl((13,13,42),(13,13,85)\bigr).
\]
\end{itemize}
\end{example}

Our next result characterizes the possible unmixed Beauville
structures for $\PGL_2(q)$.

\begin{theorem}\label{thm.Beau.PGL2q}
Let $G_1 = \PGL_2(q)$ where $3 < q = p^e$ for some odd prime number
$p$ and some positive integer $e$. Let $\tau_1 = (r_1,s_1,t_1)$,
$\tau_2 = (r_2,s_2,t_2)$ be two hyperbolic triples of integers. Then
$G_1$ admits a Beauville structure of type $(\tau_1,\tau_2)$ if and
only if the following hold:

\begin{enumerate}\renewcommand{\theenumi}{\it \roman{enumi}}
\item $G_1$ is a quotient of $T_{r_1,s_1,t_1}$ and $T_{r_2,s_2,t_2}$ with torsion-free
kernel. \\
Equivalently, $(e,\tau_1)$ and $(e,\tau_2)$ satisfy the conditions
given in Table~\ref{table.triple.PGL} in Section~\ref{sect.LR}.

\item Each of the integers
$$\gcd(r_1,r_2),\ \gcd(r_1,s_2),\ \gcd(r_1,t_2),$$
$$\gcd(s_1,r_2),\ \gcd(s_1,s_2),\ \gcd(s_1,t_2),$$
$$\gcd(t_1,r_2),\ \gcd(t_1,s_2),\ \gcd(t_1,t_2),$$
is equal to $1$ or $2$.

\item All even elements in one of the triples divide $q-1$, while all even elements
in the other triple divide $q+1$.

\item The integer $2$ appears only in a \emph{good involuting triple w.r.t $q$} (see
Definition~\ref{def.2.triple}).
\end{enumerate}
\end{theorem}

\begin{defn}\label{def.2.triple}
Let $q$ be an odd prime power. A hyperbolic triple of integers
$(r,s,2)$ is called a \emph{good involuting triple w.r.t $q$} if one
of the following holds:
\begin{itemize}
\item $q \equiv 1 \bmod 4$, and $r,s$ both divide $q-1$ but not
$(q-1)/2$;
\item $q \equiv 3 \bmod 4$, and $r,s$ both divide $q+1$ but not
$(q+1)/2$;
\item $q \equiv 1 \bmod 4$, $r$ divides $q+1$ but not $(q+1)/2$, and
$s$ is odd;
\item $q \equiv 1 \bmod 4$, $s$ divides $q+1$ but not $(q+1)/2$, and
$r$ is odd;
\item $q \equiv 3 \bmod 4$, $r$ divides $q-1$ but not $(q-1)/2$, and
$s$ is odd;
\item $q \equiv 3 \bmod 4$, $s$ divides $q-1$ but not $(q-1)/2$, and
$r$ is odd.
\end{itemize}

\end{defn}

We deduce from Theorem~\ref{thm.Beau.PGL2q} that for any odd prime
power $q\geq 5$ the group $\PGL_2(q)$ admits an unmixed Beauville
structure of type
\[
    \bigl((p,q-1,q-1),((q+1)/2,q+1,q+1)\bigr),
\]
and if $q\geq 7$ it also admits unmixed Beauville structures of
types
\[
    \bigl(((q-1)/2,q-1,q-1),(p,q+1,q+1)\bigr),
\]
and
\[
    \bigl(((q-1)/2,q-1,q-1),((q+1)/2,q+1,q+1)\bigr),
\]
(the case $q=5$ is excluded since the triple $(2,4,4)$ is not hyperbolic).

In addition, if $9 \leq q \equiv 1 \bmod 4$ then $\PGL_2(q)$ admits
unmixed Beauville structures of types
\[
    \bigl((2,p,q+1),(2,q-1,q-1)\bigr),\
    \bigl((2,p,q+1),((q-1)/2,q-1,q-1)\bigr),
\]
\[
    \bigl((2,q-1,q-1),(p,q+1,q+1)\bigr),\
    \bigl((2,q-1,q-1),((q+1)/2,q+1,q+1)\bigr),
\]
\[
    \bigl((2,(q+1)/2,q+1),(2,q-1,q-1)\bigr),\
    \bigl((2,(q+1)/2,q+1),((q-1)/2,q-1,q-1)\bigr),
\]
\[
    \bigl((2,(q+1)/2,q+1),(p,q-1,q-1)\bigr),
\]
whereas if $7 \leq q \equiv 3 \bmod 4$ then $\PGL_2(q)$ admits
unmixed Beauville structures of types
\[
    \bigl((2,p,q-1),(2,q+1,q+1)\bigr),\
    \bigl((2,p,q-1),((q+1)/2,q+1,q+1)\bigr), \quad
\]
\[
    \bigl((2,q+1,q+1),(p,q-1,q-1)\bigr),\
    \bigl((2,q+1,q+1),((q-1)/2,q-1,q-1)\bigr),
\]
and if moreover $q \geq 11$ then $\PGL_2(q)$ also admits Beauville
structures of types
\[
    \bigl((2,(q-1)/2,q-1),(2,q+1,q+1)\bigr),\
    \bigl((2,(q-1)/2,q-1),((q+1)/2,q+1,q+1)\bigr),
\]
\[
    \bigl((2,(q-1)/2,q-1),(p,q+1,q+1)\bigr).
\]
These results follow from Definition~\ref{def.2.triple} and
Table~\ref{table.good.inv.triples} in Section~\ref{sect.inv}.

\begin{example}
We list below \emph{all} the possible types for Beauville structures
for $\PGL_2(q)$ where $q=5,7,9$. This is a direct consequence of
Theorem~\ref{thm.Beau.PGL2q}, and in particular it follows from Table~\ref{table.Ex.triples.5.7.8.9}
in Section~\ref{sect.LR}.  This table describes all the hyperbolic triples satisfying condition~{\it(i)},
and one can construct all possible pairs of these triples and check whether they also satisfy
conditions~{\it(ii), (iii)} and~{\it(iv)}.
It can easily be verified by a computer, for example using \textsc{Magma}.
\begin{itemize}
\item $\PGL_2(5)$:
$$\bigl((3, 6, 6),(4, 4, 5)\bigr).$$
\item $\PGL_2(7)$:
\[
    \bigl((2, 6, 7),(2, 8, 8)\bigr), \bigl((2, 6, 7),(4, 8, 8)\bigr),
    \bigl((2, 8, 8),(6, 6, 7)\bigr), \bigl((4, 8, 8),(6, 6, 7)\bigr),
\]
\[
    \bigl((3, 6, 6),(4, 8, 8)\bigr), \bigl((2, 8, 8),(3, 6, 6)\bigr), \bigl((3, 6, 6),(7, 8, 8)\bigr).
\]
\item $\PGL_2(9)$:
\[
    \bigl((4, 8, 8),(5, 10, 10)\bigr), \bigl((2, 5, 10),(3, 8, 8)\bigr),
    \bigl((2, 8, 8 ),(5, 10, 10)\bigr), \bigl((2, 3, 10),(2, 8, 8)\bigr),
\]
\[
    \bigl((2, 5, 10),(2, 8, 8)\bigr), \bigl((2, 5, 10),(4, 8, 8)\bigr), \bigl((3, 10, 10),(4, 8, 8)\bigr),
\]
\[
    \bigl((3, 8, 8),(5, 10, 10)\bigr), \bigl((2, 8, 8),(3, 10, 10)\bigr), \bigl((2, 3, 10),(4, 8, 8)\bigr).
\]
\end{itemize}
\end{example}
\begin{example}
We list below \emph{all} the possible types of the form
$\bigl((2,r_1,s_1), (2,r_2,s_2)\bigr)$ for Beauville structures for
$\PGL_2(q)$ where $q=11,13$. This is a direct consequence of
Theorem~\ref{thm.Beau.PGL2q}, and in particular it follows from
Table~\ref{table.Ex.triples.11.13} in Section~\ref{sect.LR}, in the
same way as the previous example.
\begin{itemize}
\item $\PGL_2(11)$:
\[
    \bigl((2, 4, 12),(2, 5, 10)\bigr), \bigl((2, 4, 12),(2, 10, 11)\bigr),
    \bigl((2, 10, 11),(2, 12, 12)\bigr), \bigl((2, 5, 10),(2, 12,
    12)\bigr).
\]
\item $\PGL_2(13)$:
\[
    \bigl((2, 7, 14),(2, 12, 12)\bigr), \bigl((2, 4, 12),(2, 13, 14)\bigr),
    \bigl((2, 12, 12),(2, 13, 14)\bigr), \bigl((2, 4, 12),(2, 7,
    14)\bigr).
\]
\end{itemize}
\end{example}

\subsection*{Organization}
This paper is organized as follows. In Section~\ref{sect.pre} we
present some of the basic properties of the groups $\PSL_2(q)$ and
$\PGL_2(q)$ that are needed later. In Section~\ref{sect.surj} we
describe the results of~\cite{LR1,LR2} characterizing, for a given
$q$, the hyperbolic triangle groups which have $\PSL_2(q)$
(respectively, $\PGL_2(q)$) as quotients with torsion-free kernel,
and discuss the notions of a \emph{good $G$-order} and a \emph{good
involuting triple w.r.t $q$}. The proofs of
Theorems~\ref{thm.Beau.PSL2q} and~\ref{thm.Beau.PGL2q} are presented
in Section~\ref{sect.beau}.

\subsubsection*{Acknowledgement.}
I am thankful to Fritz Grunewald for introducing me to the
fascinating world of Beauville structures and for giving me
inspiration and motivation. He is deeply missed.

I am grateful to Alexandre Zalesski for his enlightening comments. I
am thankful to Ingrid Bauer and Fabrizio Catanese for many useful
discussions, to Matteo Penegini for many interesting discussions and
for his remarks on this manuscript, to Tatiana Bandman for patiently
answering my questions, and to Claude Marion for kindly providing me
with his recent preprints.

I am grateful to the referee for his careful reading of
this manuscript and for his valuable comments.

\section{Preliminaries}\label{sect.pre}
In this section we shall describe some well-known properties of the
groups $\PSL_2(q)$ and $\PGL_2(q)$, their elements and their
subgroups (see for example \cite{Di}, \cite[\S2.8]{Go} and
\cite[\S6]{Su}), that will be used later on.

\subsection{The groups $\PSL_2(q)$ and $\PGL_2(q)$}\label{sect.groups}

We let $\mathbb{F}_q$ denote a finite field of $q$ elements where
$q=p^e$ for some prime number $p$ and some positive integer $e$.
Recall that $\GL_2(q)$ is the group of invertible $2 \times 2$
matrices over $\mathbb{F}_q$, and $\SL_2(q)$ is the subgroup of
$\GL_2(q)$ comprising the matrices with determinant $1$. Then
$\PGL_2(q)$ and $\PSL_2(q)$ are the quotients of $\GL_2(q)$ and
$\SL_2(q)$ by their respective centers. In addition, $\PSL_2(q)$ is
simple for $q \neq 2,3$. We shall denote by $G, G_0, G_1$ the groups
$\PSL_2(q)$, $\SL_2(q)$ and $\PGL_2(q)$ respectively.

Also recall that $G$ can be viewed as a normal subgroup of $G_1$
whose index is $2$ if $p$ is odd, otherwise $G$ can be identified
with $G_1$. Let $d=\gcd(2,q-1)$. Then the orders of $G_0$, $G_1$ and
$G$ are $q(q-1)(q+1)$, $q(q-1)(q+1)$ and $q(q-1)(q+1)/d$
respectively.

Let $\mathbb{P}_1(q)$ denote the projective line over
$\mathbb{F}_q$. Then $G_1$ acts on $\mathbb{P}_1(q)$ by
\[
    \begin{pmatrix}
    a & b \\ c&d
    \end{pmatrix}: \quad
    z \mapsto \frac{az+b}{cz+d}
\]
hence, it can be identified with the group of projective
transformations on $\mathbb{P}_1(q)$. Under this identification, $G$
is the set of all transformations for which $ad-bc$ is a square in
$\mathbb{F}_q$.

\subsection{Group elements}\label{sect.elements}
One can classify the elements of $G$ and $G_1$ according to their
action on $\mathbb{P}_1(q)$. This is the same as considering the
possible Jordan forms of their pre-images. For a matrix $A \in G_0$
we will denote by $\bar{A}$ its image in $G$.

Table~\ref{table.gp.elm} lists the three types of elements according
to whether they have $0,1$ or $2$ fixed points in $\mathbb{P}_1(q)$.

\begin{table}[h]
\begin{tabular} {|c|c|c|c|}
\hline
type & action on $\mathbb{P}_1(q)$ & order in $\PGL_2(q)$ & order in $\PSL_2(q)$  \\
\hline

unipotent & fixes one point & $p$ & $p$ \\

split & fixes two points & divides $q-1$ & divides $(q-1)/d$  \\

non-split & no fixed points & divides $q+1$ & divides $(q+1)/d$ \\
\hline
\end{tabular}
\caption{Elements in $\PGL_2(q)$ and $\PSL_2(q)$}
\label{table.gp.elm}
\end{table}

Table~\ref{table.elm.PSL} describes the Jordan forms of the three
types of elements in $G$, according to whether the characteristic
polynomial $P(\la):=\la^2 - \al \la +1$ of the pre-image $A \in G_0$
(where $\al$ is the trace of $A$) has $0$, $1$ or $2$ distinct roots
in $\mathbb{F}_q$.

\begin{table}[h]
\begin{tabular} {|c|c|c|c|}
\hline
type & roots of $P(\la)$ & Jordan form in $\SL_2(\overline{\mathbb{F}}_p)$ & conjugacy classes  \\
\hline \hline

unipotent & $1$ root & $\begin{pmatrix} \pm 1 & 1 \\ 0 & \pm 1 \end{pmatrix}$ & $d$ classes in $G$ \\
 & & $\al=\pm 2$ & which unite in $G_1$ \\
\hline

split & $2$ roots & $\begin{pmatrix} a & 0 \\ 0 & a^{-1}\end{pmatrix}$ & one class in $G$ \\
      &           & where $a \in \mathbb{F}_q^*$ & for each $\al$ \\
      &           & and $a+a^{-1}=\al$ &  \\
\hline

non-split & no roots & $\begin{pmatrix} a & 0 \\ 0 & a^q\end{pmatrix}$ & one class in $G$ \\
          &          & where $a \in \mathbb{F}_{q^2}^* \setminus \mathbb{F}_q^*$ & for each $\al$\\
          &          & $a^{q+1}=1$ and $a + a^q = \al$ &  \\
\hline
\end{tabular}
\caption{Elements in $\PSL_2(q)$ and their Jordan forms}
\label{table.elm.PSL}
\end{table}

\subsection{Subgroups of $\PSL_2(q)$}\label{sect.subgroups}
Table~\ref{table.subgroups} specifies all the subgroups of
$G=\PSL_2(q)$ up to isomorphism following~\cite[Theorems 6.25 and 6.26]{Su}.

\begin{table}[h]
\begin{tabular} {|c|c|c|}
\hline
type &  maximal order & conditions \\
\hline \hline

$p$-group & $q$ & -- \\
\hline

Frobenius (Borel) & $q(q-1)/d$ & -- \\
\hline

cyclic (split) & $(q-1)/d$ & -- \\
\hline

dihedral (split) & $2(q-1)/d$ & -- \\
\hline

cyclic (non-split) & $(q+1)/d$ & -- \\
\hline

dihedral (non-split) & $2(q+1)/d$ & -- \\
\hline

$\PSL_2(q_1)$ & -- & $q=q_1^m$ $(m \in \mathbb{N})$\\
\hline

$\PGL_2(q_1)$ & -- & $q$ is odd, $q=q_1^{2m}$ $(m \in \mathbb{N})$ \\
\hline

$A_4$ & 12 &  $q$ is odd; or $q=2^e$, $e$ even \\
\hline

$S_4$ & 24 & $q^2 \equiv 1 \bmod {16}$ \\
\hline

$A_5$ & 60 & $p=5$ or $q^2 \equiv 1 \bmod {5}$ \\
\hline
\end{tabular}
\caption{Subgroups of $\PSL_2(q)$} \label{table.subgroups}
\end{table}

These subgroups can be divided into the following three classes,
following Macbeath~\cite{Mac}. The subgroups isomorphic to
$\PSL_2(q_1)$ or $\PGL_2(q_1)$ are usually called \emph{subfield}
subgroups (since $\mathbb{F}_{q_1}$ is a subfield of
$\mathbb{F}_q$). Since $A_4$, $S_4$, $A_5$ and dihedral groups
correspond to the finite triangle groups, that is, triangle groups
$T_{r,s,t}$ such that $1/r+1/s+1/t>1$, we will call them
\emph{small} subgroups. For convenience we will refer to the other
subgroups, namely subgroups of the Borel and cyclic subgroups, as
\emph{structural} subgroups.

Regarding the conjugacy classes of these subgroups, we recall that
there is a single conjugacy class in $G$ of dihedral subgroups of
order $2(q-1)/d$ (respectively, $2(q+1)/d$), so that there is a
single conjugacy class in $G$ of cyclic subgroups of order $(q-1)/d$
(respectively, $(q+1)/d$).

We also recall that for any divisor $f$ of $e$, $G$ has a
$G_1$-conjugacy class of subgroups isomorphic to $\PSL_2(p^f)$.
Moreover, if $p$ is odd and $e$ is even then $G$ has a
$G_1$-conjugacy class of subgroups isomorphic to $\PGL_2(p^{f})$ for
any $f$ dividing $e/2$.

\section{Hyperbolic triangle groups and $\PSL_2(q),\PGL_2(q)$}\label{sect.surj}

In order to characterize the possible types of an unmixed Beauville
structure for $\PSL_2(q)$ (respectively, $\PGL_2(q)$) it is crucial
to know given $q$ the hyperbolic triangle groups which have
$\PSL_2(q)$ (respectively, $\PGL_2(q)$) as quotients with
torsion-free kernel.

Given a prime power $q$ the hyperbolic triangle groups which have
$\PSL_2(q)$ (respectively, $\PGL_2(q)$) as quotients with
torsion-free kernel have been determined by Langer and
Rosenberger~\cite{LR1} and Levin and Rosenberger~\cite{LR2},
following Macbeath~\cite{Mac}. It follows that if $(r,s,t)$ is
hyperbolic, then for almost all primes $p$, there is precisely one
group of the form $\PSL_2(p^e)$ or $\PGL_2(p^e)$ which is a
homomorphic image of $T_{r,s,t}$ with torsion-free kernel. The
remaining primes $p$ satisfy that at least one of $r,s,t$ is a
multiple of $p$ which is not $p$, and for such primes, for all
positive integers $e$, neither $\PSL_2(p^e)$ nor $\PGL_2(p^e)$
contains three elements of orders $r, s$ and $t$. Recently,
Marion~\cite{Mar09} has provided another proof for the case where
$r,s,t$ are primes relying on probabilistic group theoretical
methods.

Before stating these results in \S\ref{sect.LR} we introduce some
notation in \S\ref{sect.TrOrd}. In \S\ref{sect.Mac} we present the
main results of Macbeath~\cite{Mac} and explain the concept for a
hyperbolic triple $(r,s,t)$ to be \emph{irregular w.r.t $q$}. In
Sections \S\ref{sect.inv} and \S\ref{sect.unip} we discuss the
notions of \emph{good involuting triple w.r.t $q$} and \emph{good
$G$-order} respectively.

\subsection{Orders and traces}\label{sect.TrOrd}
If $n$ is a positive integer dividing $(q-1)/d$ or $(q+1)/d$ or
equal to $p$, then $G=\PSL_2(q)$ contains an element of order $n$.
In this case, we will say that $n$ is a \emph{$G$-order}. Similarly,
if $n$ is a positive integer dividing $q-1$ or $q+1$ or equal to
$p$, then $G_1=\PGL_2(q)$ contains an element of order $n$, and we
will say that $n$ is a \emph{$G_1$-order}.

More precisely, one would like to determine the smallest positive
integer $e$ such that $\PGL_2(p^e)$ (respectively $\PSL_2(p^e)$)
contains an element of order $n$, hence we introduce the following
notation.

For a prime $p$, a positive integer $n$ coprime to $p$, and a
$k$-tuple $(n_1,\dots,n_k)$ of positive integers $n_i$ each coprime
to $p$ or equal to $p$, we let
\[
\mu_{\PGL}(p,p) = \mu_{\PSL}(p,p)=1,
\]
\[
\mu_{\PGL}(p,n)= \min \bigl\{f>0 \,:\, p^f \equiv \pm 1 \bmod{n}
\bigr\},
\]
\[
\mu_{\PSL}(p,n)= \min \bigl\{f>0 \,:\, p^f \equiv \pm 1
\bmod{(\gcd(2,n)\cdot n)}\bigr\},
\]
\[
\mu_{\PGL}(p;n_1,\dots,n_k)= \lcm
\bigl(\mu_{\PGL}(p,n_1),\dots,\mu_{\PGL}(p,n_k)\bigr)
\]
and
\[
\mu_{\PSL}(p;n_1,\dots,n_k)= \lcm
\bigl(\mu_{\PSL}(p,n_1),\dots,\mu_{\PSL}(p,n_k)\bigr).
\]

Therefore, $\mu_{\PGL}(p,n)$ (respectively $\mu_{\PSL}(p,n)$) is the
smallest positive integer $e$ such that $\PGL_2(p^e)$ (respectively
$\PSL_2(p^e)$) contains an element of order $n$. Also
$\mu_{\PGL}(p;n_1,\dots,n_k)$ (respectively
$\mu_{\PSL}(p;n_1,\dots,n_k)$) is the smallest positive integer $e$
such that $\PGL_2(p^e)$ (respectively $\PSL_2(p^e)$) contains
elements of orders $n_1,\dots,n_k$.

For any non-central matrix $A \in G_0$, its trace $\tr(A)$
determines uniquely the $G_1$-conjugacy class of $\bar{A}$ (see
Table~\ref{table.elm.PSL}), and so also the order of $\bar{A}$ is
uniquely determined by $\tr(A)$.

Hence, for a $G$-order $n$, we denote
\begin{equation}
\Tr_q(n) = \{\al \in \mathbb{F}_q :\, \al=\tr(A), A \in G_0, |\bar{A}|=n\}.
\end{equation}
It is easy to see from Table~\ref{table.elm.PSL} that for any prime
power $q$, $\Tr_q(2)=\{0\}$, $\Tr_q(3)=\{\pm 1\}$, and for any odd
$q=p^e$, $\Tr_q(p)=\{\pm 2\}$. Moreover, when $q$ is odd, then $\al
\in \Tr_q(n)$ if and only if $-\al \in \Tr_q(n)$. In fact, for any
prime power $q$ and integer $n>1$, $\Tr_q(n)$ can be effectively
computed as follows.

\begin{prop}\label{prop.Tr.q.n} Denote by $\Pri_q(n)$ the set of primitive roots of unity of
order $n$ in $\mathbb{F}_q$.
\begin{itemize}
\item Let $q=2^e$ for some positive integer $e$ and let $n>1$ be an integer, then
\begin{equation}
\Tr_q(n) = \begin{cases} \{0\} & if \ n=2 \\
\{a+a^{-1}: a \in \Pri_q(n)\} & if \ n \ divides \ q-1 \\
\{b+b^q: b \in \Pri_{q^2}(n)\} & if \ n \ divides \ q+1 \\
\emptyset & otherwise
\end{cases}
\end{equation}

\item Let $q=p^e$ for some odd prime $p$ and some positive integer $e$ and let $n>1$ be an integer, then
\begin{equation}
\Tr_q(n) = \begin{cases} \{\pm 2\} & if \ n=p \\
\{\pm (a+a^{-1}): a \in \Pri_q(2n)\} & if \ n \ divides \ \frac{q-1}{2} \\
\{\pm (b+b^q): b \in \Pri_{q^2}(2n)\} & if \ n \ divides \ \frac{q+1}{2} \\
\emptyset & otherwise
\end{cases}
\end{equation}
\end{itemize}
\end{prop}

\begin{proof} We prove the case where $q$ is odd and $n$ divides $(q-1)/2$.
The other cases are similar.

Assume first that $n$ is even. Let $a$ be a primitive root of unity
of order $2n$. Then $-a$ is also a primitive root of unity of order
$2n$, and $(-a)^n = a^n = -1$. Thus the matrices
\[
    A = \begin{pmatrix}
    a & 0 \\ 0 & a^{-1}
    \end{pmatrix}
    \quad \text{and} \quad
    -A = \begin{pmatrix}
    -a & 0 \\ 0 & -a^{-1}
    \end{pmatrix}
\]
both reduce to the same element $\bar{A} \in G$ of order $n$. Hence,
$a+a^{-1}$ and $-(a+a^{-1})$ both belong to $\Tr_q(n)$. $\Tr_q(n)$
contains only the elements of the claimed form, since it suffices to
consider only the conjugacy classes of elements of $G$, described in
Table~\ref{table.elm.PSL}.

Now assume that $n$ is odd. Then $a$ is a primitive root of unity of
order $n$ if and only if $-a$ is a primitive root of unity of order
$2n$. In this case, $(-a)^n=-a^n=-1$. Thus the matrices
\[
    A = \begin{pmatrix}
    a & 0 \\ 0 & a^{-1}
    \end{pmatrix}
    \quad \text{and} \quad
    -A = \begin{pmatrix}
    -a & 0 \\ 0 & -a^{-1}
    \end{pmatrix}
\]
both reduce to the same element $\bar{A} \in G$ of order $n$. Hence,
$a+a^{-1}$ and $-(a+a^{-1})$ both belong to $\Tr_q(n)$. Again,
considering the conjugacy classes appearing in
Table~\ref{table.elm.PSL} shows that $\Tr_q(n)$ contains only the
elements of the claimed form.
\end{proof}

\subsection{Hyperbolic triples and $\PSL_2(q),\PGL_2(q)$}\label{sect.LR}

The following theorems summarize the results in~\cite[Theorems 4.1
and 4.2]{LR1} and~\cite[Theorems 1 and 2]{LR2}, which characterize,
for a given $q$, the hyperbolic triangle groups which have
$\PSL_2(q)$ (respectively, $\PGL_2(q)$) as quotients with
torsion-free kernel. The notion of an \emph{irregular triple w.r.t
$q$} is given in Lemma~\ref{lem.irr} in \S\ref{sect.Mac}.

\begin{thm}\cite{LR1,LR2}.\label{thm.tri.PSL2q}
Given a prime $p$ and a hyperbolic triple $(r,s,t)$ of integers,
$\PSL_2(p^e)$ is a quotient of $T_{r,s,t}$ with torsion-free kernel
if and only if $(r,s,t)$ and $e$ satisfy one of the conditions given
in Table~\ref{table.triple.PSL}.

\begin{table}[h]
\begin{tabular} {|c|c|c|c|}
\hline $p$ & $(r,s,t)$ & $e$ & further conditions \\
\hline \hline

$p \geq 5$ & $(p,p,p)$ & $1$ & - \\
\hline

$p \geq 3$ & permutation of $(p,p,t')$ & $\mu_{\PSL}(p,t')$ & - \\
& $\gcd(t',p)=1$ & & \\ \hline

$p \geq 3$ & permutation of $(p,s',t')$ & $\mu_{\PSL}(p;s',t')$ & either at most one of $r,s,t$ is even, \\
& $\gcd(s' t', p)=1$ & & or: if at least two of $r,s,t$ are even \\
\cline{1-3}
$p \geq 3$ & $\gcd(r s t, p)=1$ & $\mu_{\PSL}(p;r,s,t)$ & then $(r,s,t)$ is not irregular w.r.t $p^e$  \\
\hline

$p=2$ & - & $\mu_{\PSL}(2;r,s,t)$ & -\\
\hline
\end{tabular}
\caption{Hyperbolic triangle groups $T_{r,s,t}$ which have
$\PSL_2(p^e)$ as quotients} \label{table.triple.PSL}
\end{table}
\end{thm}

\begin{thm}\cite{LR1,LR2}.\label{thm.tri.PGL2q}
Given an odd prime $p$ and a hyperbolic triple $(r,s,t)$ of
integers, $\PGL_2(p^e)$ is a quotient of $T_{r,s,t}$ with
torsion-free kernel if and only if $(r,s,t)$ and $e$ satisfy one of
the conditions given in Table~\ref{table.triple.PGL}.

\begin{table}[h]
\begin{tabular} {|c|c|c|}
\hline $(r,s,t)$ & $e$ & further conditions \\
\hline \hline
permutation of $(p,s',t')$ & $\mu_{\PSL}(p;s',t')/2$ & at least two of $r,s,t$ are even \\
$\gcd(s' t', p)=1$ & & and \\
\cline{1-2}
$\gcd(r s t, p)=1$ & $\mu_{\PSL}(p;r,s,t)/2$ & $(r,s,t)$ is irregular w.r.t $p^{2e}$ \\
\hline
\end{tabular}
\caption{Hyperbolic triangle groups $T_{r,s,t}$ which have
$\PGL_2(p^e)$ as quotients} \label{table.triple.PGL}
\end{table}
\end{thm}

The following corollary follows immediately from
Theorems~\ref{thm.tri.PSL2q} and~\ref{thm.tri.PGL2q}.

\begin{corr}\label{cor.tri.all}
Given a prime $p$ and a hyperbolic triple $(r,s,t)$ of integers,
such that each of $r,s,t$ is either coprime to $p$ or equal to $p$,
there exists a unique exponent $e$ such that $\PSL_2(p^e)$ or
$\PGL_2(p^e)$ is a quotient of $T_{r,s,t}$ with torsion-free kernel.
More precisely, let $e = \mu_{\PSL}(p;r,s,t)$ then
\begin{enumerate}\renewcommand{\theenumi}{\alph{enumi}}
\item If $(r,s,t)$ is \emph{not irregular w.r.t to $p^e$} then
$\PSL_2(p^e)$ is a quotient of $T_{r,s,t}$ with torsion-free kernel.

\item If $e$ is even and $(r,s,t)$ is \emph{irregular w.r.t to $p^e$}
then $\PGL_2(p^{e/2})$ is a quotient of $T_{r,s,t}$ with
torsion-free kernel.
\end{enumerate}
\end{corr}

\begin{example}\label{ex.gen.triples}
In Tables~\ref{table.Ex.triples.5.7.8.9}
and~\ref{table.Ex.triples.11.13} we present \emph{all} the
hyperbolic triples $(r,s,t)$ of integers such that $\PSL_2(q)$
(respectively $\PGL_2(q)$) is a quotient of $T_{r,s,t}$ with
torsion-free kernel, for $q=3,4,5,7,8,9,11,13$.
These triples were computed using \textsc{Magma}.

The \emph{irregular triples w.r.t $q^2$} are divided according to
the three cases of Lemma~\ref{lem.irr}, and among them, the
\emph{good involuting triples w.r.t $q$} are marked in {\bf bold}
(see \S\ref{sect.Mac} and \S\ref{sect.inv}).

\begin{table}[h]
\begin{tabular} {|c|c|c|c|cc|}
\hline $q$ & $G$-orders & $G_1$-orders & triples for $G$ & & triples for $G_1$ (irregular) \\
\hline \hline
$3$ & $2,3$ & $2,3,$ & None & $(\al)$ &  $(3, 4, 4)$ \\
& & $4$ &  & & \\ \hline
$4$ & $2,3,5$ & $2,3,5$ & $(2,5,5)$, $(3,3,5)$, & & None \\
& & & $(3,5,5)$, $(5,5,5)$ & & \\
\hline
$5$ & $2,3,5$ & $2,3,5,$ & $(2,5,5)$, $(3,3,5)$,
& $(\al)$ &  $(3, 4, 4)$, $(3, 4, 6)$, $(3, 6, 6)$, \\
& & $4,6$ & $(3,5,5)$, $(5,5,5)$ & & $(4, 4, 5)$, $(4, 5, 6)$, $(5, 6, 6)$ \\
& & & & $(\be)$ & $(2, 4, 6)$, $(2, 6, 6)$ \\
& & & & $(\ga)$ & $(2, 4, 5)$, ${\bf (2, 5, 6)}$ \\
\hline
$7$ & $2,3,4,7$ & $2,3,4,7,$ & $(2, 3, 7)$, $(2, 4, 7)$, & $(\al)$ & $(3, 6, 6)$, $(3, 6, 8)$, $(3, 8, 8)$, \\
& & $6,8$ & $(2, 7, 7)$, $(3, 3, 4)$, & & $(4, 6, 6)$, $(4, 6, 8)$, $(4, 8, 8)$, \\
& & & $(3, 3, 7)$, $(3, 4, 4)$, & & $(6, 6, 7)$, $(6, 7, 8)$, $(7, 8, 8)$ \\
& & & $(3, 4, 7)$, $(3, 7, 7)$, & $(\be)$ & $(2, 6, 6)$, $(2, 6, 8)$, ${\bf (2, 8, 8)}$ \\
& & & $(4, 4, 4)$, $(4, 4, 7)$, & $(\ga)$ & $(2, 3, 8 )$, $(2, 4, 6)$, $(2, 4, 8)$, \\
& & & $(4, 7, 7)$, $(7, 7, 7)$ & & ${\bf (2, 6, 7)}$, $(2, 7, 8)$ \\
\hline
$8$ & $2,3,7,9$ & $2,3,7,9$ & $(2, 3, 7)$, $(2, 3, 9)$, & & \\
& & & $(2, 7, 7)$, $(2, 7, 9)$, & & None \\
& & & $(2, 9, 9)$, $(3, 3, 7)$, & & \\
& & & $(3, 3, 9)$, $(3, 7, 7)$, & & \\
& & & $(3, 7, 9)$, $(3, 9, 9)$, & & \\
& & & $(7, 7, 7)$, $(7, 7, 9)$, & & \\
& & & $(7, 9, 9)$, $(9, 9, 9)$ & & \\
\hline
$9$ & $2,3,4,5$ & $2,3,4,5,$ & $(2, 4, 5)$, $(2, 5, 5)$, & $(\al)$ & $(3, 8, 8)$, $(3, 8, 10)$, $(3, 10, 10)$, \\
& & $8,10$ & $(3, 3, 4)$, $(3, 3, 5)$, & & $(4, 8, 8)$, $(4, 8, 10)$, $(4, 10, 10)$, \\
& & & $(3, 4, 5)$, $(3, 5, 5)$, & & $(5, 8, 8)$, $(5, 8, 10)$, $(5, 10, 10)$ \\
& & & $(4, 4, 4)$, $(4, 4, 5)$, & $(\be)$ & ${\bf (2, 8, 8)}$, $(2, 8, 10)$, $(2, 10, 10)$ \\
& & & $(4, 5, 5)$, $(5, 5, 5)$  & $(\ga)$ & $(2, 3, 8)$, ${\bf (2, 3, 10)}$, $(2, 4, 8)$, \\
& & & &  & $(2, 4, 10)$, $(2, 5, 8)$, ${\bf (2, 5, 10)}$ \\
\hline
\end{tabular}
\caption{Hyperbolic triples for $G$ and hyperbolic triples for $G_1$
(irregular w.r.t $q^2$), where $q=3,4,5,7,8,9$.}
\label{table.Ex.triples.5.7.8.9}
\end{table}

\begin{table}[h]
\begin{tabular} {|c|c|c|c|cc|}
\hline $q$ & $G$-orders & $G_1$-orders & triples for $G$ & & triples for $G_1$ (irregular) \\
\hline \hline
$11$ & $2, 3,$ & $2, 3,$ & $(2, 3, 11)$, $(2, 5, 5)$, & $(\al)$ & $(3, 4, 4)$, $(3, 4, 10)$, $(3, 4, 12)$, \\
& 5, 6, 11 & 5, 6, 11, & $(2, 5, 6)$, $(2, 5, 11)$, & & $(3, 10, 10)$, $(3, 10, 12)$, $(3, 12, 12)$, \\
& & 4, 10, 12 & $(2, 6, 6)$, $(2, 6, 11)$, & & $(4, 4, 5)$, $(4, 4, 6)$, $(4, 4, 11)$, \\
& & & $(2, 11, 11)$, $(3, 3, 5)$, & & $(4, 5, 10)$, $(4, 5, 12)$, $(4, 6, 10)$, \\
& & & $(3, 3, 6)$, $(3, 3, 11)$, & & $(4, 6, 12)$, $(4, 10, 11)$, $(4, 11, 12)$, \\
& & & $(3, 5, 5)$, $(3, 5, 6)$, & & $(5, 10, 10)$, $(5, 10, 12)$, $(5, 12, 12)$, \\
& & & $(3, 5, 11)$, $(3, 6, 6)$, & & $(6, 10, 10)$, $(6, 10, 12)$, $(6, 12, 12)$, \\
& & & $(3, 6, 11)$, $(3, 11, 11)$, & & $(10, 10, 11)$, $(10, 11, 12)$, $(11, 12, 12)$ \\
& & & $(5, 5, 5)$, $(5, 5, 6)$, & $(\be)$ & $(2, 4, 10)$, ${\bf (2, 4, 12)}$, $(2, 10, 10)$, \\
& & & $(5, 5, 11)$, $(5, 6, 6)$, & & $(2, 10, 12)$, ${\bf (2, 12, 12)}$ \\
& & & $(5, 6, 11)$, $(5, 11, 11)$, & $(\ga)$ & ${\bf (2, 3, 10)}$, $(2, 3, 12)$, $(2, 4, 5)$, \\
& & & $(6, 6, 6)$, $(6, 6, 11)$, & & $(2, 4, 6)$, $(2, 4, 11)$, ${\bf (2, 5, 10)}$, \\
& & & $(6, 11, 11)$, $(11, 11, 11)$ & & $(2, 5, 12)$, $(2, 6, 10)$, $(2, 6, 12)$, \\
& & & & & ${\bf (2, 10, 11)}$, $(2, 11, 12)$ \\
\hline
$13$ & $2, 3,$ & $2, 3,$ & $(2, 3, 7)$, $(2, 3, 13)$, & $(\al)$ & $(3, 4, 4)$, $(3, 4, 12)$, $(3, 4, 14)$, \\
& 6, 7, 13 &  6, 7, 13, & $(2, 6, 6)$, $(2, 6, 7)$, & & $(3, 12, 12)$, $(3, 12, 14)$, $(3, 14, 14)$, \\
& & $4, 12, 14$ & $(2, 6, 13)$, $(2, 7, 7)$, & & $(4, 4, 6)$, $(4, 4, 7)$, $(4, 4, 13)$, \\
& & & $(2, 7, 13)$, $(2, 13, 13)$, & & $(4, 6, 12)$, $(4, 6, 14)$, $(4, 7, 12)$, \\
& & & $(3, 3, 6)$, $(3, 3, 7)$, & & $(4, 7, 14)$, $(4, 12, 13)$, $(4, 13, 14)$, \\
& & & $(3, 3, 13)$, $(3, 6, 6)$, & & $(6, 12, 12)$, $(6, 12, 14)$, $(6, 14, 14)$, \\
& & & $(3, 6, 7)$, $(3, 6, 13)$, & & $(7, 12, 12)$, $(7, 12, 14)$, $(7, 14, 14)$, \\
& & & $(3, 7, 7)$, $(3, 7, 13)$, & & $(12, 12, 13)$, $(12, 13, 14)$, $(13, 14, 14)$ \\
& & & $(3, 13, 13)$, $(6, 6, 6)$, & $(\be)$ & ${\bf (2, 4, 12)}$, $(2, 4, 14)$, ${\bf (2, 12, 12)}$, \\
& & & $(6, 6, 7)$, $(6, 6, 13)$, & & $(2, 12, 14)$, $(2, 14, 14)$ \\
& & & $(6, 7, 7)$, $(6, 7, 13)$, & $(\ga)$ & $(2, 3, 12)$, ${\bf (2, 3, 14)}$, $(2, 4, 6)$, \\
& & & $(6, 13, 13)$, $(7, 7, 7)$, & & $(2, 4, 7)$, $(2, 4, 13)$, $(2, 6, 12)$, \\
& & & $(7, 7, 13)$, $(7, 13, 13)$, & & $(2, 6, 14)$, $(2, 7, 12)$, ${\bf (2, 7, 14)}$, \\
& & & $(13, 13, 13)$ & & $(2, 12, 13)$, ${\bf (2, 13, 14)}$ \\
\hline
\end{tabular}
\caption{Hyperbolic triples for $G$ and hyperbolic triples for $G_1$
(irregular w.r.t $q^2$), where $q=11,13$}
\label{table.Ex.triples.11.13}
\end{table}
\end{example}

\subsection{Generating triples and irregular triples}\label{sect.Mac}
Macbeath~\cite{Mac} classified the pairs of elements in $G$ in a way
which makes it easy to decide what kind of subgroup they generate.
He called a triple $(A,B,C)$ of elements in $G$ (respectively $G_0$)
such that $ABC=1$ a \emph{$G$-triple} (respectively
\emph{$G_0$-triple}). So if $(A,B,C)$ is a $G_0$-triple then $(\bar
A, \bar B, \bar C)$ is a $G$-triple.

By \cite[Theorem 1]{Mac}, for any $(\al,\be,\ga) \in
\mathbb{F}_q^3$, there exists a $G_0$-triple $(A,B,C)$ such that
$A$, $B$ and $C$ have respective traces $\al$, $\be$ and $\ga$.
Hence, if $(r,s,t)$ is a triple of $G$-orders then there exists a
$G$-triple $(\bar A, \bar B, \bar C)$ such that $\bar{A}$, $\bar{B}$
and $\bar{C}$ have respective orders $r$, $s$ and $t$.

Macbeath~\cite{Mac} called a $G_0$-triple $(A,B,C)$ \emph{singular}
if its corresponding traces $(\al,\be,\ga)$ satisfy the equality
\begin{equation}\label{eq.sing}
    \al^2+\be^2+\ga^2-\al\be\ga-4 = 0.
\end{equation}
Moreover, by \cite[Theorem 2]{Mac}, a $G_0$-triple $(A,B,C)$ is
singular if and only if the corresponding $G$-triple $(\bar A, \bar
B, \bar C)$ satisfies that $\langle \bar A, \bar B \rangle$ is a
\emph{structural} subgroup of $G$.

Observe that if $\langle \bar A, \bar B \rangle$ is a \emph{small}
subgroup, then the corresponding orders $(r,s,t)$ of $(\bar A, \bar
B, \bar C)$ satisfy that either two of $r,s,t$ equal to $2$ or
$r,s,t \in \{2,3,4,5\}$, but the converse might not be true. Indeed,
if $(\bar A, \bar B, \bar C)$ is a $G$-triple of respective orders
$(r,s,t)$ such that $(r,s,t)$ is hyperbolic and $r,s,t \in
\{2,3,4,5\}$ then $\langle \bar A, \bar B \rangle$ is not
necessarily a \emph{small} subgroup (see~\cite{LR2}).

In fact, when $(r,s,t)$ is a hyperbolic triple of $G$-orders, then
there is enough freedom in choosing the traces $\al \in \Tr_q(r)$,
$\be \in \Tr_q(s)$ and $\ga \in \Tr_q(t)$, such that
Equation~\eqref{eq.sing} does not hold, and in addition, they do not
correspond to a $G$-triple which generates a \emph{small} subgroup
(see \cite[Lemma 3.3]{LR1} and \cite{LR2}). Therefore, if $(r,s,t)$
is a hyperbolic triple of $G$-orders then there exists a $G$-triple
$(A, B, C)$ such that $A$, $B$ and $C$ have respective orders $r$,
$s$ and $t$, and moreover, $\langle A,B \rangle$ is a
\emph{subfield} subgroup of $G$.

When $p$ is odd and $e$ is even there are $G$-triples $(A,B,C)$
which generate a projective special linear subgroup $\PSL_2(q_1)$
(respectively projective general linear group $\PGL_2(q_1)$), where
$\mathbb{F}_{q_1}$ (respectively $\mathbb{F}_{q_1^2}$) is a subfield
of $\mathbb{F}_q$ (see Table~\ref{table.subgroups}).

If $q=q_1^2$ and $(A,B,C)$ is a $G$-triple that generates a subgroup
isomorphic to $\PGL_2(q_1)$ then exactly one of $(A,B,C)$ lies in
$\PSL_2(q_1)$, and we say that $(A,B,C)$ is an \emph{irregular
$G$-triple} (see \cite[\S9]{Mac}). On the other hand, if $(r,s,t)$
is a hyperbolic triple of $\PGL_2(q_1)$-orders then in particular it is
a hyperbolic triple of $G$-orders, hence there exists a
$G$-triple $(A,B,C)$ such that $A$, $B$ and $C$ have respective
orders $r$, $s$ and $t$. Consequently, $(r,s,t)$ is said to be
\emph{irregular w.r.t $q$} if $(A,B,C)$ is an irregular $G$-triple.
Langer and Rosenberger determined in \cite[Lemma 3.5]{LR1} the
necessary and sufficient condition for $(r,s,t)$ to be irregular
w.r.t $q$.

\begin{lemma}\cite{LR1}.\label{lem.irr}
Let $q=p^e$ be an odd prime power and let $(r,s,t)$ be a hyperbolic
triple of integers such that $\gcd(rst,p)=1$ or one of $r,s,t$ is
equal to $p$ and the two others are coprime to $p$. Then $(r,s,t)$
is \emph{irregular w.r.t $q$} if up to a permutation $(r',s',t')$ of
$(r,s,t)$ one of the following cases holds:

{\bf Case $(\al)$:} \begin{itemize}
\item $r',s',t' > 2$,
\item $r'$, $s'$ and $e=\mu_{\PSL}(p;r',s',t')$ are all even,
\item both $\mu_{\PGL}(p,r')$ and $\mu_{\PGL}(p,s')$ divide $\frac{e}{2}$,
\item both $\mu_{\PSL}(p,r')$ and $\mu_{\PSL}(p,s')$ do not divide $\frac{e}{2}$,
\item $\mu_{\PSL}(p,t')$ divides $\frac{e}{2}$.
\end{itemize}

{\bf Case $(\be)$:} \begin{itemize}
\item $r',s'> 2$ and $t'=2$,
\item $r'$, $s'$ and $e=\mu_{\PSL}(p;r',s')$ are all even,
\item both $\mu_{\PGL}(p,r')$ and $\mu_{\PGL}(p,s')$ divide $\frac{e}{2}$,
\item both $\mu_{\PSL}(p,r')$ and $\mu_{\PSL}(p,s')$ do not divide
$\frac{e}{2}$.
\end{itemize}

{\bf Case $(\ga)$:} \begin{itemize}
\item $r',s' > 2$, and $t'=2$,
\item $r'$ and $e=\mu_{\PSL}(p;r',s')$ are even,
\item $\mu_{\PGL}(p,r')$ divides $\frac{e}{2}$,
\item $\mu_{\PSL}(p,r')$ does not divide $\frac{e}{2}$,
\item $\mu_{\PSL}(p,s')$ divides $\frac{e}{2}$.
\end{itemize}
\end{lemma}

Case $(\be)$ is the same as case $(\al)$ except that $t'=2$. Observe
that the difference between the last two cases is that an irregular
$G$-triple $(A,B,C)$ in case $(\be)$ contains an involution which
belongs to $\PSL_2(q_1)$, while in case $(\ga)$ the involution
belongs to $\PGL_2(q_1) \setminus \PSL_2(q_1)$. We will therefore
investigate in detail irregular triples containing involutions in
\S\ref{sect.inv}.

As an example, in Tables~\ref{table.Ex.triples.5.7.8.9}
and~\ref{table.Ex.triples.11.13} we present all the irregular
triples w.r.t $q^2$, for $q=3,5,7,9,11,13$, divided according to the
above cases. These triples were computed using \textsc{Magma}.

\subsection{Irregular triples containing involutions}\label{sect.inv}
In this section we consider irregular $G$-triples $(A,B,C)$ where
$C$ is an involution. We give a numerical criterion to decide
whether all the elements of even order in this triple are of the
same type, either split or non-split. Such triples, called
\emph{``good involuting triples w.r.t $q$''} (see
Definition~\ref{def.2.triple}), are needed in the classification of
Beauville structures for $G_1$ which include involutions, in
Theorem~\ref{thm.Beau.PGL2q}{\it (iv)}.

Recall that all the involutions in $G=\PSL_2(q)$ are conjugate to
the image of the matrix $\begin{pmatrix} 0 & 1 \\ -1 &
0\end{pmatrix}$. They are unipotent if $p=2$, split if $(q-1)/2$ is
even, namely if $q \equiv 1 \bmod 4$, and non-split if $(q+1)/2$ is
even, namely if $q \equiv 3 \bmod 4$. Moreover, if $q$ is odd, then
there is exactly one $G_1$-conjugacy class of involutions in $G_1
\setminus G$. They are split if $(q-1)/2$ is odd, namely if $q
\equiv 3 \bmod 4$, and non-split if $(q+1)/2$ is odd, namely if $q
\equiv 1 \bmod 4$ (see Tables~\ref{table.gp.elm}
and~\ref{table.elm.PSL}).

\begin{prop}\label{prop.2.triple}
Assume that $q$ is an odd prime power and let $(A,B,C)$ be an
irregular $\PSL_2(q^2)$-triple of respective orders $(r,s,2)$. Then
one of the following holds:
\begin{itemize}
\item $q \equiv 1 \bmod 4$, $C$ is split and $(r,s,2)$ is in case $(\be)$.
\item $q \equiv 1 \bmod 4$, $C$ is non-split and $(r,s,2)$ is in case $(\ga)$.
\item $q \equiv 3 \bmod 4$, $C$ is non-split and $(r,s,2)$ is in case $(\be)$.
\item $q \equiv 3 \bmod 4$, $C$ is split and $(r,s,2)$ is in case $(\ga)$.
\end{itemize}
\end{prop}

\begin{proof}
Recall that by Lemma~\ref{lem.irr}, an irregular
$\PSL_2(q^2)$-triple $(A,B,C)$ of respective orders $(r,s,2)$ is in
case $(\be)$ if and only if $C \in \PSL_2(q)$, whereas it is in case
$(\ga)$ if and only if $C \in \PGL_2(q) \setminus \PSL_2(q)$ (up to
conjugation). The claim now follows from the above observation.
\end{proof}

\begin{corr}\label{cor.rs2}
Assume that $q$ is an odd prime power and let $(A,B,C)$ be an
irregular $\PSL_2(q^2)$-triple of respective orders $(r,s,2)$ with
$r>2$ even.

If $q \equiv 1 \bmod 4$, then
\begin{itemize}
\item Both $A$ and $C$ are split if and only if $r$ divides $q-1$ and
$(r,s,2)$ is in case $(\be)$.
\item Both $A$ and $C$ are non-split if and
only if $r$ divides $q+1$ and $(r,s,2)$ is in case $(\ga)$.
\end{itemize}

If $q \equiv 3 \bmod 4$, then
\begin{itemize}
\item Both $A$ and $C$ are split if and only if $r$ divides $q-1$ and
$(r,s,2)$ is in case $(\ga)$.
\item Both $A$ and $C$ are non-split if and
only if $r$ divides $q+1$ and $(r,s,2)$ is in case $(\be)$.
\end{itemize}
\end{corr}

The proof of part~{\it (iv)} of Theorem~\ref{thm.Beau.PGL2q} relies
on the following corollary which is a consequence of
Corollary~\ref{cor.rs2}. Recall that the notion of a \emph{``good
involuting triple w.r.t $q$''} was given in
Definition~\ref{def.2.triple}.

\begin{corr}\label{cor.2.triple}
Let $q$ be an odd prime power and let $(A,B,C)$ be an irregular
$\PSL_2(q^2)$-triple of respective orders $(r,s,2)$. Then all the
elements of even order in this triple are of the same type (either
split or non-split) if and only if $(r,s,2)$ is a \emph{good
involuting triple w.r.t $q$} (see Definition~\ref{def.2.triple}).
\end{corr}
\begin{proof}
Again, let $G=\PSL_2(q)$ and $G_1=\PGL_2(q)$. Without loss of
generality we may assume that $A \in G_1\setminus G$ and thus $r>2$
is even. By Corollary~\ref{cor.rs2} one of the following necessarily
holds:
\begin{itemize}
\item $q \equiv 1 \bmod 4$, $r$ divides $q-1$ and $(r,s,2)$ is in case $(\be)$.
In case $(\be)$, both $A,B \in G_1 \setminus G$, and so $s$ is also
even. Since $A,B$ are split, then $r,s$ both divide $q-1$ but not
$(q-1)/2$.

\item $q \equiv 1 \bmod 4$, $r$ divides $q+1$ and $(r,s,2)$ is in case $(\ga)$.
Since $A \in G_1\setminus G$ then $r$ does not divide $(q+1)/2$. In
case $(\ga)$, $B \in G$, and so $s$ is a $G$-order. If $s$ is even
then $s$ divides $(q-1)/2$, and so $B$ is split. But $A$ is
non-split, yielding a contradiction. Hence, $s$ is necessarily odd.

\item $q \equiv 3 \bmod 4$, $r$ divides $q-1$ and $(r,s,2)$ is in case $(\ga)$.
Since $A \in G_1\setminus G$ then $r$ does not divide $(q-1)/2$. In
case $(\ga)$, $B \in G$, and so $s$ is a $G$-order. If $s$ is even
then $s$ divides $(q+1)/2$, and so $B$ is non-split. But $A$ is
split, yielding a contradiction. Hence, $s$ is necessarily odd.

\item $q \equiv 3 \bmod 4$, $r$ divides $q+1$ and $(r,s,2)$ is in
case $(\be)$. In case $(\be)$, both $A,B \in G_1 \setminus G$, and
so $s$ is also even. Since $A,B$ are non-split, then $r,s$ both
divide $q+1$ but not $(q+1)/2$.
\end{itemize}
\end{proof}

As an example, in Tables~\ref{table.Ex.triples.5.7.8.9}
and~\ref{table.Ex.triples.11.13} we mark in {\bf bold} all the
\emph{good involuting triples w.r.t $q$} where $q=5,7,9,11,13$.
These triples were computed using \textsc{Magma}.

\begin{example}
Table~\ref{table.good.inv.triples} presents some general examples
for \emph{good involuting triples w.r.t $q$}.

\begin{table}[h]
\begin{tabular} {|cc|c|c|c|}
$q$ & & case & elements of even order & good involuting triple w.r.t $q$ \\
\hline \hline
$q \equiv 1 \bmod 4$ & $q>5$ & $(\be)$ & split & $(2,q-1,q-1)$ \\
\hline
$q \equiv 1 \bmod 4$ & $q \geq 5$ & $(\ga)$ & non-split & $(2,p,q+1)$ \\
& $q>5$ & & & $(2,(q+1)/2,(q+1))$\\
\hline
$q \equiv 3 \bmod 4$ & $q \geq 7$ & $(\ga)$ & split & $(2,p,q-1)$ \\
& $q > 7$ &  &  & $(2,(q-1)/2,q-1)$ \\
\hline
$q \equiv 3 \bmod 4$ & $q \geq 7$ & $(\be)$ & non-split & $(2,q+1,q+1)$ \\
\hline
\end{tabular}
\caption{Good involuting triples w.r.t $q$}
\label{table.good.inv.triples}
\end{table}
Note that in case $(\be)$ we exclude $q=5$ since the triple $(2,4,4)$ is not hyperbolic,
and in case $(\ga)$ the second triple is excluded when $q=5$ or $7$
since the triple $(2,3,6)$ is not hyperbolic.
\end{example}

\subsection{Generating triples containing unipotents}\label{sect.unip}
In this section we consider $G$-triples $(A,B,C)$ where $A$ and $B$
are unipotent elements and $C$ is not unipotent. We give a numerical
criterion on the order of $C$ to decide whether $A$ is $G$-conjugate
to $B$, which is called a \emph{``good $G$-order''} (see
Definition~\ref{defn.good.G.order}). Such triples are needed in the
classification of Beauville structures for $G$ which include
unipotents, in Theorem~\ref{thm.Beau.PSL2q}{\it (ii)}.

Assume that $q$ is odd and consider the following matrices in
$G_0=\SL_2(q)$:
\[
    U_1 = \begin{pmatrix} 1 & 1 \\ 0 & 1 \end{pmatrix}, \quad
    U_{-1} = \begin{pmatrix} -1 & 1 \\ 0 & -1 \end{pmatrix},
\]
\[
    U'_1 = XU_1X^{-1} = \begin{pmatrix} 1 & x^2 \\ 0 & 1 \end{pmatrix} \in
    G_0, \quad
    U'_{-1} = XU_{-1}X^{-1} = \begin{pmatrix} -1 & x^2 \\ 0 & -1 \end{pmatrix} \in
    G_0,
\]
where $x \in \mathbb{F}_{q^2} \setminus \mathbb{F}_q$ satisfies that
$x^2 \in \mathbb{F}_q$ and $X=\begin{pmatrix} x & 0 \\ 0 & x^{-1}
\end{pmatrix} \in \SL_2(q^2)$.

The following proposition and its corollary are immediate
observations.

\begin{prop}\label{prop.unip.G0}
Assume that $q$ is odd. Then, for any $A \in G_0$, $XAX^{-1} \in
G_0$. Moreover,
\begin{itemize}
\item If $A \ne I$ and $\tr(A)=2$ then $A$ is $G_0$-conjugate to either $U_1$ or $U'_1$.
\item If $A \ne -I$ and $\tr(A)=-2$ then $A$ is $G_0$-conjugate to either $U_{-1}$ or $U'_{-1}$.
\end{itemize}
\end{prop}
\begin{proof}
Indeed, if $A=\begin{pmatrix} a & b \\ c & d \end{pmatrix} \in G_0$
then $XAX^{-1} = \begin{pmatrix} a & bx^2 \\ cx^{-2} & d \end{pmatrix} \in G_0$.
Moreover, $U_1$ is not $G_0$-conjugate to $U'_1=XU_1X^{-1}$,
since $x^2$ is not a square of some element in $\mathbb{F}_q$.
Hence, any $I \ne A \in G_0$ with $\tr(A)=2$ is $G_0$-conjugate to either $U_1$ or $U'_1$
(see Table~\ref{table.elm.PSL}).
\end{proof}

\begin{corr}\label{corr.unip.G}
Assume that $q$ is odd. Then, for any $\bar{A} \in G$,
$\bar{X}\bar{A}\bar{X}^{-1} \in G$. If, moreover, $\bar{A}$ is
unipotent then it is $G$-conjugate to either $\bar{U}_1$ or
$\bar{U'}_1=\bar{X}\bar{U}_1\bar{X}^{-1}$. In addition,
\begin{itemize}
\item If $q \equiv 1 \bmod 4$ then $\bar{U}_1$ and $\bar{U}_{-1}$ are $G$-conjugate.
\item If $q \equiv 3 \bmod 4$ then $\bar{U}_1$ and $\bar{U'}_{-1}$ are $G$-conjugate.
\end{itemize}
\end{corr}

The following lemma is needed later in Section~\ref{sect.conj},
to decide whether a unipotent element $\bar{A} \in G$ is $G$-conjugate
to some power of $\bar{U}_1$ or not.

\begin{lemma}\label{lem.unip.i}
Assume that $p$ is odd and let $q=p^e$. Let $\bar{A} \in G$ be a
unipotent element.
\begin{itemize}
\item If $e$ is odd, then there exists some $0 < i < p$ such that
$\bar{A}^i$ is $G$-conjugate to $\bar{U}_1$.
\item If $e$ is even, then for every $0 < i < p$,
$\bar{A}^i$ is $G$-conjugate to $\bar{A}$.
\end{itemize}
\end{lemma}
\begin{proof}
Consider the set $I=\{ i: 0< i <p \}$. Observe that if $p$ is odd
and $e$ is even then all the elements in $I$ are squares in
$\mathbb{F}_q$. If $p$ is odd and $e$ is odd, then half of the
elements in $I$ are squares in $\mathbb{F}_q$ and half are
non-squares.

Hence, if $e$ is odd then there exists some $i \in I$ such that
$U_1^i$ is $G_0$-conjugate to $U'_1$ and so $\bar{U}_1^i$ is
$G$-conjugate to $\bar{U'}_1$. If $e$ is even then for every $i \in
I$, $U_1^i$ is $G_0$-conjugate to $U_1$ and so $\bar{U}_1^i$ is
$G$-conjugate to $\bar{U}_1$.
\end{proof}

We now consider $G$-triples $(A,B,C)$ such that $A$ and $B$ are
unipotent elements and $C$ is not unipotent.

\begin{prop}\label{prop.G0.unip.triple}
Assume that $q$ is odd. Let $(A,B,C)$ be a $G_0$-triple such that
$A,B \ne \pm I$ and $\tr(A),\tr(B)\in \{\pm 2\}$. Denote
$\ga=\tr(C)$.
\begin{enumerate}
\item If $\tr(A)=\tr(B)$, then $A$ is $G_0$-conjugate to $B$ if and
only if $2-\ga$ is a square in $\mathbb{F}_q$.
\item If $\tr(A)=-\tr(B)$, then $A$ is $G_0$-conjugate to $-B$ if and
only if $2+\ga$ is a square in $\mathbb{F}_q$.
\end{enumerate}
\end{prop}
\begin{proof}
Without loss of generality we may assume that $A=U_1$.

\begin{enumerate}
\item If $B=MU_1M^{-1}$ for some matrix $M=\begin{pmatrix} a & b \\ c
& d \end{pmatrix} \in G_0$, then
\[
    \ga = \tr(C) = \tr(AB) = \tr(U_1MU_1M^{-1}) =
    \tr \begin{pmatrix}  1-ac-c^2 & 1+ac+a^2 \\ -c^2 & 1+ac \end{pmatrix} = 2-c^2,
\]
and so $2-\ga$ is a square in $\mathbb{F}_q$.

\item[] If $B=MU'_1M^{-1}$ for some matrix $M \in G_0$, then
\[
    \ga = \tr(C) = \tr(AB) = \tr(U_1MU'_1M^{-1}) = 2-x^2c^2,
\]
and since $x \in \mathbb{F}_{q^2} \setminus \mathbb{F}_q$ then
$2-\ga$ is a non-square in $\mathbb{F}_q$.

\item Similarly, if $B=M(-U_{1})M^{-1}$ for some matrix $M \in G_0$,
then
\[
    \ga = \tr(C) = \tr(AB) = \tr(U_1M(-U_{1})M^{-1}) = -2+c^2,
\]
and so $2+\ga$ is a square in $\mathbb{F}_q$.

\item[] If $B=M(-U'_{1})M^{-1}$ for some matrix $M \in G_0$, then
\[
    \ga = \tr(C) = \tr(AB) = \tr(U_1M(-U'_{1})M^{-1}) = -2+x^2c^2,
\]
and so $2+\ga$ is a non-square in $\mathbb{F}_q$.
\end{enumerate}
\end{proof}

Therefore, in order to decide whether in a $G$-triple $(A,B,C)$ of
respective orders $(p,p,t)$, $t \ne p$, $A$ is $G$-conjugate to $B$,
one needs to determine whether for $\ga \in \Tr_q(t)$, $2-\ga$ or
$2+\ga$ is a square in $\mathbb{F}_q$. In the following we prove
that this is equivalent to decide whether $t$ is a \emph{good
$G$-order} or not. Recall that $G=\PSL_2(q)$, $G_0=\SL_2(q)$ and the
notion of a \emph{``good $G$-order''} was given in
Definition~\ref{defn.good.G.order}.

\begin{prop}\label{prop.good.G.order}
Assume that $q$ is odd. Let $C \in G_0$, $\ga=\tr(C)$ and
$t=|\bar{C}|$. Assume that $\ga \ne \pm 2$ (or equivalently, $t \ne
p$). Then one of $2+\ga$, $2-\ga$ is a square in $\mathbb{F}_q$ if
and only if $t$ is a \emph{good $G$-order} (see
Definition~\ref{defn.good.G.order}).
\end{prop}
\begin{proof}
As $t \ne p$ is a $G$-order then $t$ divides either $(q-1)/2$ or
$(q+1)/2$.

If $t$ divides $(q-1)/2$ then $\ga=a+a^{-1}$ or $\ga=-(a+a^{-1})$,
for some primitive root of unity $a$ of order $2t$ in $\mathbb{F}_q$
(see Proposition~\ref{prop.Tr.q.n}). Hence,
\[
    \{2+\ga, 2-\ga\} = \{a+2+a^{-1},(-a)+2+(-a)^{-1}\}.
\]
Therefore, $2+\ga$ or $2-\ga$ is a square in $\mathbb{F}_q$ if and
only if $a=c^2$ or $-a=c^2$ for some $c \in \mathbb{F}_q$. Indeed,
$a+2+a^{-1}$ is a square if and only if $(a+1)^2/a$ is a square if
and only if $a$ is a square in $\mathbb{F}_q$, and similarly for
$(-a)+2+(-a)^{-1}$.

Now, one the following necessarily holds:
\begin{itemize}
\item If $t$ is even then $-a$ is also a primitive root of unity
of order $2t$. Hence, $a$ is a square in $\mathbb{F}_q$ if and only
if $-a$ is a square is $\mathbb{F}_q$. In addition, $a$ is a square
in $\mathbb{F}_q$ if and only if $\mathbb{F}_q$ contains a primitive
root of unity of order $4t$, namely, if and only if $4t$ divides
$q-1$.

\item If $t$ is odd and $q \equiv 1 \bmod 4$ then $4t$ divides
$q-1$, and so $\mathbb{F}_q$ contains a primitive root of unity of
order $4t$. Thus $a$, which is a primitive root of unity of order
$2t$, is a square in $\mathbb{F}_q$, as required.

\item If $t$ is odd and $q \equiv 3 \bmod 4$ then $\mathbb{F}_q$
contains a primitive root of unity of order $2t$ but does not
contain a primitive root of unity of order $4t$, and so, $a$ is a
non-square in $\mathbb{F}_q$. However, $-a$ is a primitive root of
unity of order $t$, and so, it is necessarily a square in
$\mathbb{F}_q$, as required.
\end{itemize}
In conclusion, $a=c^2$ or $-a=c^2$ for some $c \in \mathbb{F}_q$ if
and only if either $t$ is odd and divides $q-1$ or $t$ is even and
$4t$ divides $q-1$.

\medskip

If $t$ divides $(q+1)/2$ then $\ga=a+a^q$ or $\ga=-(a+a^q)$, for
some primitive root of unity $a$ of order $2t$ in $\mathbb{F}_{q^2}$
(see Proposition~\ref{prop.Tr.q.n}). Hence,
\[
    \{2+\ga,2-\ga\} = \{a+2+a^q,(-a)+2+(-a)^q\}.
\]
Therefore, $2+\ga$ or $2-\ga$ is a square in $\mathbb{F}_q$ if and
only if $a=c^2$ or $-a=c^2$ for some $c \in \mathbb{F}_{q^2}$
satisfying $c^{q+1}=1$.

Now, one the following necessarily holds:
\begin{itemize}
\item If $t$ is even then $-a$ is also a primitive root of unity
of order $2t$. Hence, $a=c^2$ for some $c \in \mathbb{F}_{q^2}$
satisfying $c^{q+1}=1$ if and only if $-a=b^2$ for some $b \in
\mathbb{F}_{q^2}$ satisfying $b^{q+1}=1$. This is equivalent to the
condition that $4t$ divides $q+1$.

\item If $t$ is odd and $q \equiv 3 \bmod 4$ then $\mathbb{F}_{q^2}$
contains a primitive root of unity $b$ of order $4t$ satisfying
$b^{q+1}=1$. Hence, $a=c^2$ for some $c \in \mathbb{F}_{q^2}$
satisfying $c^{q+1}=1$, as required.

\item If $t$ is odd and $q \equiv 1 \bmod 4$ then $\mathbb{F}_{q^2}$
does not contain a primitive root of unity $c$ of order $4t$
satisfying $c^{q+1}=1$. However, in this case, $-a=b^2$ for some $b
\in \mathbb{F}_{q^2}$ satisfying $b^{q+1}=1$, as required.
\end{itemize}
In conclusion, $a=c^2$ or $-a=c^2$ for some $c \in \mathbb{F}_{q^2}$
satisfying $c^{q+1}=1$ if and only if either $t$ is odd and divides
$q+1$ or $t$ is even and $4t$ divides $q+1$.
\end{proof}

\begin{corr}\label{corr.G.unip.triple}
Assume that $q=p^e$ for some odd prime $p$ and some positive integer
$e$. Let $(A,B,C)$ be a $G$-triple of respective orders $(p,p,t)$,
$t \ne p$. Then $A$ is $G$-conjugate to $B$ if and only if $t$ is a
\emph{good $G$-order}.
\end{corr}
\begin{proof}
Let $(A,B,C)$ be a $G_0$-triple and assume that its image in $G$,
$(\bar{A},\bar{B},\bar{C})$ has respective orders $(p,p,t)$, $t \ne
p$. Denote $\ga=\tr(C)$. Then $\bar{A}$ and $\bar{B}$ are unipotent
if and only if $A,B \ne \pm I$ and $\tr(A),\tr(B)\in \{\pm 2\}$.
Moreover, $\bar{A}$ and $\bar{B}$ are $G$-conjugate if and only if
either $\tr(A)=\tr(B)$ and $A$ and $B$ are $G_0$-conjugate or
$\tr(A)=-\tr(B)$ and $A$ and $-B$ are $G_0$-conjugate. From
Proposition~\ref{prop.G0.unip.triple} we deduce that $\bar{A}$ and
$\bar{B}$ are $G$-conjugate if and only if $2-\ga$ or $2+\ga$ is a
square in $\mathbb{F}_q$. By Proposition~\ref{prop.good.G.order},
the latter is equivalent to $t$ being a good $G$-order.
\end{proof}

\begin{lemma}\label{lem.good.unip.triple}
Assume that $q=p^e$ where $p$ is odd and $5 \leq q \neq 9$. There
exists a $G$-triple $(A,B,C)$ of respective orders $(p,p,t)$, $t \ne
p$, such that $\langle A, B \rangle=G$ and $A$ is $G$-conjugate to
$B$ if and only if $e=\mu_{\PSL}(p;t)$ and $t$ is a \emph{good
$G$-order}.
\end{lemma}
\begin{proof}
Let $(A,B,C)$ be a $G$-triple of respective orders $(p,p,t)$, $t \ne
p$. If $\langle A, B \rangle=G$ then Theorem~\ref{thm.tri.PSL2q}
implies that $e=\mu_{\PSL}(p;t)$, and if moreover $A$ is
$G$-conjugate to $B$ then Corollary~\ref{corr.G.unip.triple} implies
that $t$ is a good $G$-order.

If $e=\mu_{\PSL}(p;t)$ then there exists a $G$-triple $(A,B,C)$ of
respective orders $(p,p,t)$ (see Section~\ref{sect.Mac}). If
moreover $t$ is a good $G$-order then $A$ is $G$-conjugate to $B$,
by Corollary~\ref{corr.G.unip.triple}.

We now use the methodology described in Section~\ref{sect.Mac}. Let
$\ga \in \Tr_{p^e}(t)$. Observe that Equation~\eqref{eq.sing} is
equivalent in this case to $(\ga \pm 2)^2= 0$. Since $t \ne p$ then
$\ga \ne \pm 2$, and so this equality does not hold, implying that
$\langle A, B \rangle$ is not a structural subgroup,
by~\cite[Theorem 2]{Mac}. As $e=\mu_{\PSL}(p;t)$, it follows from
Table~\ref{table.Ex.triples.5.7.8.9} that if $5< q \ne 9$ then
either $p>5$; or $p=5$ and $e>1$ implying that $t \ne 2,3,5$; or
$p=3$ and $e>2$ implying that $t>5$. Therefore, $\langle A,B
\rangle$ cannot be a small subgroup. If $q=5$ then $\langle A,B
\rangle \cong A_5 = G$ as required. In addition, $(A,B,C)$ is
clearly not an irregular $G$-triple. The condition that
$e=\mu_{\PSL}(p;t)$ now ensures that $\langle A,B \rangle = G$.
\end{proof}

\begin{rem}\label{rem.9.335}
In the case $G=\PSL_2(9)$ one needs to consider the $G$-orders $4$
and $5$.
\begin{itemize}
\item $4$ is not a good $G$-order, and so, if $(A,B,C)$ is a
$G$-triple of respective orders $(3,3,4)$ then $A$, $B$ are not
$G$-conjugate.

\item $5$ is a good $G$-order. However, if $(A,B,C)$ is a $G$-triple
of respective orders $(3,3,5)$ and $A$ is $G$-conjugate to $B$, then
one can verify that $\langle A,B \rangle \cong A_5$ is a small
subgroup of $G$ (see also \cite[\S2, Theorem 8.4]{Go}).
\end{itemize}
\end{rem}

\begin{example} Table~\ref{table.good.G.orders} presents for $p=5, 7, 11, 13$ \emph{all} the
$G$-orders $t \ne p$ such that $e=\mu_{\PSL}(p;t) \in \{1,2\}$,
divided according to whether they are \emph{good $G$-orders} or not.
They were computed using \textsc{Magma}.
\begin{table}[h]
\begin{tabular} {|c|c|c||c|c|c|}
& good $G$-orders & not good $G$-orders & & good $G$-orders & not
good $G$-orders \\
\hline \hline
$q=5$ & $3$ & $2$ & $q=25$ & $6, 13$ & $4, 12$ \\
\hline
$q=7$ & $2, 3$ & $4$ & $q=49$ & $5,6,12,25$ & $8,24$ \\
\hline
$q=11$ & $3,5$ & $2,6$ & $q=121$ & $10,15,30,61$ & $4,12,20,60$ \\
\hline
$q=13$ & $3,7$ & $2,6$ & $q=169$ & $5,14,17,21,42,85$ & $4,12,28,84$ \\
\hline
\end{tabular}
\caption{Good $G$-orders for $q=5, 7, 11, 13, 25, 49, 121, 169$}
\label{table.good.G.orders}
\end{table}
\end{example}

\section{Beauville Structures for $\PSL_2(q)$ and $\PGL_2(q)$}\label{sect.beau}
In this section we prove Theorems~\ref{thm.Beau.PSL2q}
and~\ref{thm.Beau.PGL2q}.

\subsection{Cyclic groups}
The following elementary lemma is needed for the proof of
Theorems~\ref{thm.Beau.PSL2q} and~\ref{thm.Beau.PGL2q}.

\begin{lemma}\label{lem.cyclic}
Let $\cC$ be a finite cyclic group, and let $x$ and $y$ be
non-trivial elements in $\cC$. If the orders of $x$ and $y$ are not
relatively prime, then there exist some integers $k$ and $l$ such
that $x^k=y^l \neq 1$.
\end{lemma}
\begin{proof}
Let $a$ and $b$ denote the orders of $x$ and $y$ respectively and
set $c=\gcd(a,b)$. Note that by assumption $c \neq 1$. Also write
$a=a'c$ and $b=b'c$ where $\gcd(a',b')=1$, so that $x^{a'}$ and
$y^{b'}$ have order $c$.

Observe that $\cC$ has a unique (cyclic) subgroup of order $c$, and
let $z$ be a generator of this subgroup. Thus,
\[
    \langle x^{a'} \rangle = \langle z \rangle = \langle y^{b'} \rangle.
\]
Therefore, there exist some integers $k$ and $l$ such that
\[
    x^{a'k} = z = y^{b'l} \ne 1,
\]
where the latter inequality follows from the fact that $z$ is of
order $c > 1$.
\end{proof}

\subsection{Elements and conjugacy classes in $\PSL_2(q)$ and $\PGL_2(q)$}
\label{sect.conj}
Let $H = \PSL_2(q)$ or $\PGL_2(q)$ where $q = p^e$ for some prime
$p$ and some positive integer $e$. In this section, we determine the
elements $A_1,A_2$ of $H$ such that $\Sigma(A_1) \cap \Sigma(A_2) =
\{1\}$, where for elements $A_1, \dots ,A_m$ of $H$,
\[
     \Sigma(A_1,\dots,A_m)= \bigcup_{A \in H} \bigcup^{\infty}_{j=1}
     \{A A_1^j A^{-1},\dots, A A_m^j A^{-1}\}.
\]

Given two triples $(A_1,A_2,A_3)$ and $(B_1,B_2,B_3)$ of $H$, this
will enable us to determine whether $\Sigma(A_1,A_2,A_3) \cap
\Sigma(B_1,B_2,B_3) = \{1\}$ which is a necessary condition for $H$
to admit an unmixed Beauville structure (see
Definition~\ref{defn.beau}{\it(iii)}). Indeed the condition
$\Sigma(A_1,A_2,A_3) \cap \Sigma(B_1,B_2,B_3) = \{1\}$ is equivalent
to the condition $$\Sigma(A_i) \cap \Sigma(B_j) = \{1\} \quad
\forall 1 \leq i,j \leq 3.$$

\begin{lemma}\label{lem.PSL.sigma}
Let $G = \PSL_2(q)$ where $q = p^e$ for some prime number $p$ and
some positive integer $e$. Let $A_1,A_2 \in G$. Then $\Sigma(A_1)
\cap \Sigma(A_2) = \{1\}$ if and only if one of the following
occurs:

\begin{enumerate}\renewcommand{\theenumi}{\alph{enumi}}
\item The orders $|A_1|$ and $|A_2|$ are relatively prime.
\item The prime $p$ is odd, $e$ is even, $|A_1|=|A_2|=p$ and $A_1,A_2$
are not $G$-conjugate.
\end{enumerate}
\end{lemma}
\begin{proof}
If the orders of $A_1$ and $A_2$ are relatively prime then every two
non-trivial powers $A_1^i$ and $A_2^j$ have distinct orders, and
thus $\Sigma(A_1) \cap \Sigma(A_2) = \{1\}$ as required.

Now, assume that the orders of $A_1$ and $A_2$ are not relatively
prime.

Observe that $(q-1)/d$ and $(q+1)/d$ are coprime, where
$d=\gcd(q-1,2)$. Thus, if there exists some prime $r \neq p$ which
divides both $|A_1|$ and $|A_2|$, then $r$ divides exactly one of
$(q-1)/d$ or $(q+1)/d$. Hence, $|A_1|$ and $|A_2|$ both divide
exactly one of $(q-1)/d$ or $(q+1)/d$, and so $A_1$ and $A_2$ are
$G$-conjugate to two elements $C_1$ and $C_2$ which belong to the
same cyclic group, either of order $(q-1)/d$ or of order $(q+1)/d$
(see Section~\ref{sect.subgroups}). Lemma~\ref{lem.cyclic} now
implies that there exist some integers $i$ and $j$ such that
$A_1^{i} \neq 1$ and $A_2^{j} \neq 1$ are $G$-conjugate, and so
$\Sigma(A_1) \cap \Sigma(A_2) \neq \{1\}$.

It now remains to treat the case where $|A_1|=|A_2|=p$, that is when
$A_1$ and $A_2$ are unipotent elements. If $p=2$ then necessarily
$A_1$ and $A_2$ are $G$-conjugate and so $\Sigma(A_1) \cap
\Sigma(A_2) \neq \{1\}$. However, when $p$ is odd then $A_1$ and
$A_2$ are not necessarily $G$-conjugate (see
Table~\ref{table.elm.PSL} and Section~\ref{sect.unip}).

If $p$ is odd and $e$ is odd then by Lemma~\ref{lem.unip.i}, there
exist some integers $i$ and $j$ such that $A_1^{i} \neq 1$ and
$A_2^{j} \neq 1$ are $G$-conjugate, and so $\Sigma(A_1) \cap
\Sigma(A_2) \neq \{1\}$. If $p$ is odd and $e$ is even then by
Lemma~\ref{lem.unip.i}, for any two integers $1 \leq i,j <p$,
$A_1^i$ and $A_2^j$ are $G$-conjugate to $A_1$ and $A_2$
respectively. Thus, $\Sigma(A_1) \cap \Sigma(A_2) = \{1\}$ if and
only if $A_1$ and $A_2$ are not $G$-conjugate.
\end{proof}

\begin{prop}\label{prop.sigma.unip}
Let $G = \PSL_2(q)$ where $q = p^e$ for some odd prime $p$ and even
integer $e$. Take some $x \in \mathbb{F}_{q^2} \setminus
\mathbb{F}_q$ such that $x^2 \in \mathbb{F}_q$, and let $X\in
\PSL_2(q^2)$ denote the image of the matrix $\begin{pmatrix} x & 0
\\ 0 & x^{-1} \end{pmatrix} \in \SL_2(q^2)$.

Let $A_1,A_2,B_1,B_2 \in G$ such that $|A_1|=|A_2|=|B_1|=|B_2|=p$.
Then, $XA_2X^{-1},$ $XB_2X^{-1} \in G$, and $A_1$ is $G$-conjugate to either $A_2$ or $XA_2X^{-1}$.
Moreover,

\begin{enumerate}\renewcommand{\theenumi}{\it \roman{enumi}}
\item Either $\Sigma(A_1) \cap \Sigma(A_2) = \{1\}$ or $\Sigma(A_1) \cap \Sigma(XA_2X^{-1}) = \{1\}$.
\item $\Sigma(A_1,B_1) \cap \Sigma(A_2) = \{1\}$ or
$\Sigma(A_1,B_1) \cap \Sigma(XA_2X^{-1}) = \{1\}$ \\ if and only if
$A_1$ is $G$-conjugate to $B_1$.
\item $\Sigma(A_1,B_1) \cap \Sigma(A_2,B_2) = \{1\}$ or
$\Sigma(A_1,B_1) \cap \Sigma(XA_2X^{-1},XB_2X^{-1}) = \{1\}$ \\ if
and only if $A_1$ is $G$-conjugate to $B_1$ and $A_2$ is
$G$-conjugate to $B_2$.
\end{enumerate}
\end{prop}

\begin{proof}
By Corollary~\ref{corr.unip.G}, $XA_2X^{-1},XB_2X^{-1} \in G$,
$A_1$ is $G$-conjugate to either $A_2$ or $XA_2X^{-1}$, and
$B_1$ is $G$-conjugate to either $B_2$ or $XB_2X^{-1}$.

{\it(i)} If $A_1$ and $A_2$ are not $G$-conjugate then by Lemma~\ref{lem.PSL.sigma},
$\Sigma(A_1) \cap \Sigma(A_2) = \{1\}$. On the other hand,
if $A_1$ and $A_2$ are $G$-conjugate, then $A_1$ and $XA_2X^{-1}$ are not $G$-conjugate,
and again by Lemma~\ref{lem.PSL.sigma}, $\Sigma(A_1) \cap \Sigma(XA_2X^{-1}) = \{1\}$.

{\it(ii)} Assume that $A_1$ and $B_1$ are $G$-conjugate.
If $A_2$ is not $G$-conjugate to $A_1$ (and to $B_1$) then by Lemma~\ref{lem.PSL.sigma},
$\Sigma(A_1) \cap \Sigma(A_2) = \{1\}$ (and $\Sigma(B_1) \cap \Sigma(A_2) = \{1\}$),
implying $\Sigma(A_1,B_1) \cap \Sigma(A_2) = \{1\}$.
Otherwise, $XA_2X^{-1}$ is not $G$-conjugate to $A_1$ (and to $B_1$) and so, similarly,
$\Sigma(A_1,B_1) \cap \Sigma(XA_2X^{-1}) = \{1\}$.

Now assume that $A_1$ and $B_1$ are not $G$-conjugate. In this case, $A_2$ is $G$-conjugate
to either $A_1$ or $B_1$.
If $A_2$ is $G$-conjugate to $A_1$ then  $XA_2X^{-1}$ is $G$-conjugate to $B_1$,
and so by Lemma~\ref{lem.PSL.sigma}, $\Sigma(A_1) \cap \Sigma(A_2) \ne \{1\}$ and
$\Sigma(B_1) \cap \Sigma(XA_2X^{-1}) \ne \{1\}$.
Otherwise, $A_2$ is $G$-conjugate to $B_1$ and so $XA_2X^{-1}$ is $G$-conjugate to $A_1$,
thus, similarly, $\Sigma(B_1) \cap \Sigma(A_2) \ne \{1\}$ and
$\Sigma(A_1) \cap \Sigma(XA_2X^{-1}) \ne \{1\}$.

{\it(iii)} Assume that $A_1$ is $G$-conjugate to $B_1$ and $A_2$ is $G$-conjugate to $B_2$.
If $A_2$ (and $B_2$) are not $G$-conjugate to $A_1$ (and $B_1$) then by Lemma~\ref{lem.PSL.sigma},
$\Sigma(A_1,B_1) \cap \Sigma(A_2,B_2) = \{1\}$.
Otherwise, $A_2$ (and $B_2$) are $G$-conjugate to $A_1$ (and $B_1$), and so,
$XA_2X^{-1}$ (and $XB_2X^{-1}$) are not $G$-conjugate to $A_1$ (and $B_1$), hence, similarly,
$\Sigma(A_1,B_1) \cap \Sigma(XA_2X^{-1},XB_2X^{-1}) = \{1\}$.

If $A_1$ is not $G$-conjugate to $B_1$, then by {\it(ii)},
$\Sigma(A_1,B_1) \cap \Sigma(A_2) \ne \{1\}$ and $\Sigma(A_1,B_1)
\cap \Sigma(XA_2X^{-1}) \ne \{1\}$. Similarly, if $A_2$ is not
$G$-conjugate to $B_2$, then by {\it(ii)}, $\Sigma(A_1) \cap
\Sigma(A_2,B_2) \ne \{1\}$ and $\Sigma(A_1) \cap
\Sigma(XA_2X^{-1},XB_2X^{-1}) \ne \{1\}$.
\end{proof}

\begin{lemma}\label{lem.PGL.sigma}
Let $G_1 = \PGL_2(q)$ where $q = p^e$ for some odd prime number $p$
and some positive integer $e$. Let $A_1,A_2 \in G_1$. Then
$\Sigma(A_1) \cap \Sigma(A_2) = \{1\}$ if and only if one of the
following occurs:

\begin{enumerate}\renewcommand{\theenumi}{\alph{enumi}}
\item The orders $|A_1|$ and $|A_2|$ are relatively prime.
\item $A_1$ is split, $A_2$ is non-split and $\gcd(|A_1|,|A_2|)=2$.
\item $A_1$ is non-split, $A_2$ is split and $\gcd(|A_1|,|A_2|)=2$.
\end{enumerate}
\end{lemma}
\begin{proof}
If $\gcd(|A_1|,|A_2|)=1$ then any two non-trivial powers $A_1^i$ and
$A_2^j$ have distinct orders, thus $\Sigma(A_1) \cap \Sigma(A_2) =
\{1\}$, as required.

If $A_1$ is split and $A_2$ is non-split, then necessarily
$\gcd(|A_1|,|A_2|) \leq 2$, since $\gcd(q-1,q+1)=2$. In this case,
any non-trivial power of $A_1$ is a split element, while any
non-trivial power of $A_2$ is a non-split element, and so they are
not $G_1$-conjugate, implying that $\Sigma(A_1) \cap \Sigma(A_2) =
\{1\}$, as required.

If $\gcd(|A_1|,|A_2|)=2$ and both $A_1$ and $A_2$ are split
(respectively non-split) elements, then $A_1$ and $A_2$ are
$G_1$-conjugate to two elements $C_1$ and $C_2$ which belong to the
same cyclic group of order $q-1$ (respectively $q+1$).
Lemma~\ref{lem.cyclic} now implies that there exist some integers
$i$ and $j$ such that $A_1^{i}\neq 1$ and $A_2^{j}\neq 1$ are
$G_1$-conjugate, and so $\Sigma(A_1) \cap \Sigma(A_2) \neq \{1\}$.

If $|A_1|=|A_2|=p$, then $A_1$ and $A_2$ are unipotent, and so they
are $G_1$-conjugate, implying that $\Sigma(A_1) \cap \Sigma(A_2)
\neq \{1\}$.

Otherwise, $\gcd(|A_1|,|A_2|)=r$, where $2<r\ne p$, and so $r$
divides exactly one of $q-1$ or $q+1$, implying that $|A_1|$ and
$|A_2|$ both divide exactly one of $q-1$ or $q+1$. Hence, $A_1$ and
$A_2$ are $G_1$-conjugate to two elements $C_1$ and $C_2$ which
belong to the same cyclic group, either of order $q-1$ or of order
$q+1$. Lemma~\ref{lem.cyclic} implies again that there exist some
integers $i$ and $j$ such that $A_1^{i} \neq 1$ and $A_2^{j} \neq 1$
are $G_1$-conjugate, and so $\Sigma(A_1) \cap \Sigma(A_2) \neq
\{1\}$.
\end{proof}

\subsection{Proof of Theorem~\ref{thm.Beau.PSL2q}}\label{sect.proof.Beau.PSL2q}

\subsubsection*{The conditions are sufficient}
Let $\tau_1=(r_1,s_1,t_1)$ and $\tau_2=(r_2,s_2,t_2)$ be two
hyperbolic triples of integers. Assume that $G=\PSL_2(q)$ is a
quotient of the triangle groups $T_{r_1,s_1,t_1}$ and
$T_{r_2,s_2,t_2}$ with torsion-free kernel. Then one can find
elements $A_1,B_1,C_1,A_2,B_2,C_2$ in $G$ of orders
$r_1,s_1,t_1,r_2,s_2,t_2$ respectively, such that
$A_1B_1C_1=1=A_2B_2C_2$ and $\langle A_1,B_1 \rangle = G = \langle
A_2,B_2 \rangle$, and so conditions~{\it(i)} and~{\it(ii)} of
Definition~\ref{defn.beau} are fulfilled.

Moreover, the condition that $r_1s_1t_1$ is coprime to $r_2s_2t_2$
implies that each of $r_1,s_1,t_1$ is coprime to each of
$r_2,s_2,t_2$, and so by Lemma~\ref{lem.PSL.sigma},
$\Sigma(A_1,B_1,C_1) \cap \Sigma(A_2,B_2,C_2) = \{1\}$, hence
condition~{\it(iii)} of Definition~\ref{defn.beau} is fulfilled,
thus $\bigl((A_1,B_1,C_1),(A_2,B_2,C_2)\bigr)$ is an unmixed
Beauville structure of type $(\tau_1,\tau_2)$.

It is left to consider the case where $p$ is odd, $e$ is even,
$q=p^e>9$ and $\gcd(r_1s_1t_1,r_2s_2t_2) \in \{p,p^2\}$, which can
be reduced to the following three cases.

\begin{enumerate}
\item $r_1=r_2=p$ and $s_1,s_2,t_1,t_2 \neq p$.\\
Let $X$ be as in Proposition~\ref{prop.sigma.unip}, and denote
$A'_2=XA_2X^{-1}$, $B'_2=XB_2X^{-1}$ and $C'_2=XC_2X^{-1}$. Then
also $A'_2B'_2C'_2=1$ and $\langle A'_2,B'_2 \rangle = G$. By
Proposition~\ref{prop.sigma.unip},
$$\text{either } \Sigma(A_1) \cap \Sigma(A_2) = \{1\} \text{ or } \Sigma(A_1) \cap \Sigma(A'_2) = \{1\}.$$
Moreover, since $p$ is coprime to $s_2t_2$ then by Lemma~\ref{lem.PSL.sigma},
$$\Sigma(A_1) \cap \Sigma(B_2,C_2) = \{1\},\ \Sigma(A_1) \cap \Sigma(B'_2,C'_2) = \{1\}.$$
Similarly, since $s_1t_1$ is coprime to $ps_2t_2$, then by Lemma~\ref{lem.PSL.sigma},
$$\Sigma(B_1,C_1) \cap \Sigma(A_2,B_2,C_2) = \{1\},\ \Sigma(B_1,C_1) \cap \Sigma(A'_2,B'_2,C'_2) = \{1\}.$$
Therefore, either $\Sigma(A_1,B_1,C_1) \cap \Sigma(A_2,B_2,C_2) = \{1\}$
or $\Sigma(A_1,B_1,C_1) \cap \Sigma(A'_2,B'_2,C'_2) = \{1\}.$

\item $r_1=r_2=s_1=p$ and $s_2,t_1,t_2 \neq p$. \\
Let $(A'_2,B'_2,C'_2)$ be as in Case (1). By the assumption of the theorem,
in this case, $t_1$ is a \emph{good $G$-order}. Hence by Lemma~\ref{lem.good.unip.triple},
there exist $A_1,B_1,C_1 \in G$ of respective orders $p,p,t_1$ such
that $A_1$ is $G$-conjugate to $B_1$, $A_1B_1C_1=1$ and $\langle
A_1,B_1 \rangle = G$. By Proposition~\ref{prop.sigma.unip},
$$\text{either } \Sigma(A_1,B_1) \cap \Sigma(A_2) = \{1\}\
\text{or } \Sigma(A_1,B_1) \cap \Sigma(A'_2) = \{1\}.$$
Moreover, since $p$ is coprime to $s_2t_2$, then by Lemma~\ref{lem.PSL.sigma},
$$\Sigma(A_1,B_1) \cap \Sigma(B_2,C_2) = \{1\},\ \Sigma(A_1,B_1) \cap \Sigma(B'_2,C'_2) = \{1\}.$$
Similarly, since $t_1$ is coprime to $ps_2t_2$, then by Lemma~\ref{lem.PSL.sigma},
$$\Sigma(C_1) \cap \Sigma(A_2,B_2,C_2) = \{1\},\ \Sigma(C_1) \cap \Sigma(A'_2,B'_2,C'_2) = \{1\},$$
Therefore, either $\Sigma(A_1,B_1,C_1) \cap \Sigma(A_2,B_2,C_2) = \{1\}$
or $\Sigma(A_1,B_1,C_1) \cap \Sigma(A'_2,B'_2,C'_2) = \{1\}.$

\item $r_1=r_2=s_1=s_2=p$ and $t_1,t_2 \neq p$.\\
Let $(A_1,B_1,C_1)$ be as in Case (2). By the assumption of the theorem,
in this case, $t_1$ and $t_2$ are \emph{good $G$-orders}. Hence by Lemma~\ref{lem.good.unip.triple},
there exist $A_2,B_2,C_2 \in G$ of respective orders $p,p,t_2$ such
that $A_2$ is $G$-conjugate to $B_2$, $A_2B_2C_2=1$ and $\langle
A_2,B_2 \rangle = G$. Again, we denote $A'_2=XA_2X^{-1}$,
$B'_2=XB_2X^{-1}$ and $C'_2=XC_2X^{-1}$. Then also $A'_2B'_2C'_2=1$
and $\langle A'_2,B'_2 \rangle = G$. By Proposition~\ref{prop.sigma.unip},
$$\text{either } \Sigma(A_1,B_1) \cap \Sigma(A_2,B_2) = \{1\}\
\text{or } \Sigma(A_1,B_1) \cap \Sigma(A'_2,B'_2) = \{1\}.$$
Moreover, since $p$, $t_1$ and $t_2$ are pairwise coprime, then by Lemma~\ref{lem.PSL.sigma},
$$\Sigma(A_1,B_1) \cap \Sigma(C_2) = \{1\},\ \Sigma(A_1,B_1) \cap \Sigma(C'_2) = \{1\},$$
$$\Sigma(C_1) \cap \Sigma(A_2,B_2) = \{1\},\ \Sigma(C_1) \cap \Sigma(A'_2,B'_2) = \{1\},$$
$$\Sigma(C_1) \cap \Sigma(C_2) = \{1\},\ \Sigma(C_1) \cap \Sigma(C'_2) = \{1\}.$$
Therefore, either $\Sigma(A_1,B_1,C_1) \cap \Sigma(A_2,B_2,C_2) = \{1\}$
or $\Sigma(A_1,B_1,C_1) \cap \Sigma(A'_2,B'_2,C'_2) = \{1\}.$
\end{enumerate}

We conclude that in these three cases, either $\bigl((A_1,B_1,C_1),
(A_2,B_2,C_2)\bigr)$ or $\bigl((A_1,B_1,C_1)$,
$(A'_2,B'_2,C'_2)\bigr)$ is an unmixed Beauville structure of type
$(\tau_1,\tau_2)$.

\subsubsection*{The conditions are necessary}
Assume that the group $G=\PSL_2(q)$ admits an unmixed Beauville
structure of type $(\tau_1,\tau_2)$, where $\tau_i=(r_i,s_i,t_i)$
for $i=1,2$. Then there exist $A_1,B_1,C_1,A_2,B_2,C_2$ in $G$ of
orders $r_1,s_1,t_1,r_2,s_2,t_2$ respectively, such that
$A_1B_1C_1=1=A_2B_2C_2$ and $\langle A_1,B_1 \rangle = G = \langle
A_2,B_2 \rangle$, implying that $G$ is a quotient of the triangle
groups $T_{r_1,s_1,t_1}$ and $T_{r_2,s_2,t_2}$ with torsion-free
kernel, and so condition~{\it(i)} is necessary.

Moreover, $\Sigma(A_1,B_1,C_1) \cap \Sigma(A_2,B_2,C_2) = \{1\}$,
and so by Lemma~\ref{lem.PSL.sigma}, if $p=2$ or $e$ is odd, then
each of $r_1,s_1,t_1$ is necessarily coprime to each of
$r_2,s_2,t_2$, implying that $r_1s_1t_1$ is coprime to $r_2s_2t_2$.

If $p$ is odd and $e$ is even then, by Lemma~\ref{lem.PSL.sigma},
$\gcd(r_1,r_2)=1$ or $p$, $\gcd(r_1,s_2)=1$ or $p$,
$\gcd(r_1,t_2)=1$ or $p$, $\gcd(s_1,r_2)=1$ or $p$,
$\gcd(s_1,s_2)=1$ or $p$, $\gcd(s_1,t_2)=1$ or $p$,
$\gcd(t_1,r_2)=1$ or $p$, $\gcd(t_1,s_2)=1$ or $p$, and
$\gcd(t_1,t_2)=1$ or $p$. Moreover, it is not possible that
$r_1=s_1=t_1=p$ (respectively $r_2=s_2=t_2=p$), since in this case
$e=1$, by Theorem~\ref{thm.tri.PSL2q}. Thus, $g=\gcd(r_1s_1t_1,
r_2s_2t_2) \in \{1,p,p^2\}$.

If moreover, $q=p^e>9$, $p$ divides $g$ and $\tau_i = (p,p,t_i)$ ($i
\in \{1,2\}$) then $t_i \neq p$ and the condition that
$\Sigma(A_1,B_1,C_1) \cap \Sigma(A_2,B_2,C_2) = \{1\}$ implies that
$A_i$ is $G$-conjugate to $B_i$, by
Proposition~\ref{prop.sigma.unip}. We now deduce from
Corollary~\ref{corr.G.unip.triple} that $t_i$ is a \emph{good
$G$-order}.

If $q=9$ then it follows from a careful observation of the possible
$G$-triples in Table~\ref{table.Ex.triples.5.7.8.9} and
Remark~\ref{rem.9.335} that necessarily $g=1$.

\medskip

In fact, if $p$ is odd, $e$ is even and $q=p^e>9$, then $\PSL_2(q)$
\emph{always} admits unmixed Beauville structures of type
$\bigl((p,p,t_1),(p,p,t_2)\bigr)$ for certain $t_1,t_2$. In the
following lemma we explicitly construct such a structure.

\begin{lemma}\label{lem.ppt}
Let $3 < q = p^e$ for some odd prime number $p$ and some positive
integer $e$. Then $\PSL_2(q^2)$ admits an unmixed Beauville
structure of type $\bigl( (p,p,t_1),(p,p,t_2)\bigr)$ for certain
$t_1$ dividing $(q^2-1)/2$ and $t_2$ dividing $(q^2+1)/2$.
\end{lemma}

\begin{proof}
As $q>3$, the following Remark~\ref{rem.sets.Fq} shows that there exist
some $b,c \in \mathbb{F}_{q^2}$ such that $b^2,c^2 \in \mathbb{F}_{q^2} \setminus
\mathbb{F}_q$, $c^2-4$ is a square in $\mathbb{F}_{q^2}$ and $b^2-4$
is a non-square in $\mathbb{F}_{q^2}$. Let $x$ be a generator of the
multiplicative group $\mathbb{F}_{q^2}^*$ and set $d=b/x$.

Define the following matrices
\begin{align*}
   A_1 = \begin{pmatrix} 1 & 1 \\ 0 & 1
   \end{pmatrix}, &\quad
   A_2 = \begin{pmatrix} 1 & x \\ 0 & 1
   \end{pmatrix}, \\
   g_1 = \begin{pmatrix} 1 & 0 \\ c & 1
   \end{pmatrix}, &\quad
   g_2 = \begin{pmatrix} 1 & 0 \\ d & 1
   \end{pmatrix}, \\
   B_1 = g_1A_1g_1^{-1}= \begin{pmatrix} -c + 1 & 1 \\ -c^2 & c + 1
   \end{pmatrix}, &\quad
   B_2 = g_2A_2g_2^{-1}= \begin{pmatrix} -dx + 1 & x \\ -d^2x & dx+1
   \end{pmatrix}, \\
   C_1 = (A_1B_1)^{-1}= \begin{pmatrix} c + 1 & -c - 2 \\ c^2 & -c^2 - c
   + 1 \end{pmatrix}, &\quad
   C_2 = (A_2B_2)^{-1}= \begin{pmatrix} dx + 1 & -dx^2 - 2x \\ d^2x & -d^2x^2 - dx + 1
   \end{pmatrix}.
\end{align*}

In this case, $|\bar{A_1}|=|\bar{A_2}|=|\bar{B_1}|=|\bar{B_2}|=p$,
and we denote by $t_1$ and $t_2$ respectively the orders of
$\bar{C_1}$ and $\bar{C_2}$. Moreover, $\bar{A_1}$ and $\bar{B_1}$ are conjugate in $\PSL_2(q^2)$,
$\bar{A_2}$ and $\bar{B_2}$ are conjugate in $\PSL_2(q^2)$, whereas
$\bar{A_1}$ and $\bar{A}_2$ are not conjugate in $\PSL_2(q^2)$ (see Section~\ref{sect.unip}).
Now, one needs to verify that
$\bigl((\bar{A_1},\bar{B_1},\bar{C_1}),(\bar{A_2},\bar{B_2},\bar{C_2})\bigr)$
is an unmixed Beauville structure for $\PSL_2(q^2)$.

\begin{enumerate}\renewcommand{\theenumi}{\it \roman{enumi}}
\item By the construction, $\bar{A_1}\bar{B_1}\bar{C_1} = 1 =
\bar{A_2}\bar{B_2}\bar{C_2}$.

\item Observe that $\tr(C_1) =2-c^2$ and $\tr(C_2) =2-d^2x^2 = 2-b^2$ both belong to
$\mathbb{F}_{q^2}\setminus \mathbb{F}_q$, as $c^2$ and $b^2$ both
belong to $\mathbb{F}_{q^2}\setminus \mathbb{F}_q$. Hence, neither
$\bar{C_1}$ nor $\bar{C_2}$ is conjugate to some element of
$\PSL_2(q)$. We now use the methodology described in
Section~\ref{sect.Mac}. Let $i \in \{1,2\}$. Since $\tr(A_i)=
\tr(B_i)=2$ but $\tr(C_i) \ne \pm 2$, the triple $(A_i,B_i,C_i)$ is
not singular, implying that $\langle \bar{A}_i, \bar{B}_i \rangle$
is not a structural subgroup, by~\cite[Theorem 2]{Mac}.

As $t_i$ is not an order of an element in $\PSL_2(q)$, it follows
from Table~\ref{table.Ex.triples.5.7.8.9} that if $q^2>9$, then
either $p>5$; or $p=5$ and $t_i \ne 2,3,5$; or $p=3$ and $t_i>5$.
Therefore, $\langle \bar{A}_i,\bar{B}_i \rangle$ cannot be a small
subgroup. In addition, $(\bar{A}_i,\bar{B}_i,\bar{C}_i)$ is not an
irregular $\PSL_2(q^2)$-triple. Therefore, $\langle
\bar{A}_i,\bar{B}_i \rangle = \PSL_2(q^2)$ for $i \in \{1,2\}$.

\item The characteristic polynomial of $C_1$ is $\la^2-(2-c^2)+1$,
and its discriminant equals $c^2(c^2-4)$, which is a square in
$\mathbb{F}_{q^2}$, thus $\bar{C_1}$ is split and so $t_1$ divides
$(q^2-1)/2$. Similarly, the characteristic polynomial of $C_2$ is
$\la^2-(2-b^2)+1$, and its discriminant equals $b^2(b^2-4)$, which
is a non-square in $\mathbb{F}_{q^2}$, thus $\bar{C_2}$ is non-split
and so $t_2$ divides $(q^2+1)/2$. By Lemma~\ref{lem.PSL.sigma},
$\Sigma(\bar{A_1},\bar{B_1},\bar{C_1}) \cap
\Sigma(\bar{A_2},\bar{B_2},\bar{C_2}) = \{1\}$, since $t_1$ and
$t_2$ are coprime, $\bar{A_1}$ and $\bar{A_2}$ are not conjugate
in $\PSL_2(q^2)$, and so also $\bar{B_1}$ and $\bar{B_2}$ are not conjugate
in $\PSL_2(q^2)$.
\end{enumerate}
\end{proof}


\begin{rem}\label{rem.sets.Fq}
Recall that if $q$ is odd then an element in $\PSL_2(q)$ is
non-split if and only if the characteristic polynomial
$P(\la):=\la^2 - \al \la +1$ of its pre-image $A \in \SL_2(q)$
(where $\al=\tr(A)$) has no distinct roots in $\mathbb{F}_q$, or
equivalently, the discriminant $\al^2-4$ is a non-square in
$\mathbb{F}_q$. Thus, by~\cite[Lemma 2]{Mac}, $\#\{b \in
\mathbb{F}_{q}: b^2-4 \text{ is a non-square} \} = (q-1)/2$ and
$\#\{c \in \mathbb{F}_{q}: c^2-4 \text{ is a square} \} = (q+1)/2$.

Therefore, $\#\{c \in \mathbb{F}_{q^2}: c^2-4 \text{ is a square} \}
= (q^2+1)/2$, and $\#\{b \in \mathbb{F}_{q^2}: b^2-4 \text{ is a
non-square} \} = (q^2-1)/2$. In addition, $\#\{c \in
\mathbb{F}_{q^2}: c^2 \in \mathbb{F}_q\}=2q-1$, and if $c^2 \in
\mathbb{F}_q$ then also $c^2-4 \in \mathbb{F}_q$ is a square in
$\mathbb{F}_{q^2}$. Hence,
$$\#\{c \in \mathbb{F}_{q^2}: c^2 \notin \mathbb{F}_q,\
c^2-4 \text{ is a square} \} = (q^2+1)/2 -(2q-1) = (q^2-4q+3)/2,$$
and
$$\#\{b \in \mathbb{F}_{q^2}: b^2 \notin \mathbb{F}_q,\
b^2-4 \text{ is a non-square} \} = (q^2-1)/2.$$
\end{rem}

\subsection{Proof of Theorem~\ref{thm.Beau.PGL2q}}\label{sect.proof.Beau.PGL2q}
\subsubsection*{The conditions are necessary}
Assume that the group $G_1=\PGL_2(q)$ admits an unmixed Beauville
structure of type $\bigr( (r_1,s_1,t_1),(r_2,t_2,s_2) \bigl)$. Then
there exist $A_1,B_1,C_1,A_2,B_2,C_2$ in $G_1$ of orders
$r_1,s_1,t_1,r_2,s_2,t_2$ respectively, such that
$A_1B_1C_1=1=A_2B_2C_2$ and $\langle A_1,B_1 \rangle = G_1 = \langle
A_2,B_2 \rangle$, implying that $G_1$ is a quotient of the triangle
groups $T_{r_1,s_1,t_1}$ and $T_{r_2,s_2,t_2}$ with torsion-free
kernel, and so condition~{\it(i)} is necessary.

Therefore, we may assume that $(r_1,s_1,t_1)$ and $(r_2,s_2,t_2)$
are hyperbolic triples, and moreover they are \emph{irregular w.r.t
$q^2$} (see \S\ref{sect.Mac}).

If, for example, $\gcd(r_1,r_2)>2$, then Lemma~\ref{lem.PGL.sigma}
implies that $\Sigma(A_1) \cap \Sigma(A_2)$ is non-trivial,
contradicting $\Sigma(A_1,B_1,C_1) \cap \Sigma(A_2,B_2,C_2) =
\{1\}$. Hence, condition~{\it(ii)} is necessary.

Since $(r_1,s_1,t_1)$ and $(r_2,s_2,t_2)$ are hyperbolic and
irregular, both of them must contain at least two even
integers, one of which is greater than $2$. Hence, we may assume
that $r_1,r_2$ are even and that $r_1,r_2 > 2$. If both $r_1,r_2$
divide $q-1$ (respectively $q+1$) then both $A_1,A_2$ are split
(respectively non-split) and by Lemma~\ref{lem.PGL.sigma},
$\Sigma(A_1) \cap \Sigma(A_2) \neq \{1\}$, yielding a contradiction.

Hence, we may assume that $r_1$ divides $q-1$ and $r_2$ divides
$q+1$, and so $A_1$ is split and $A_2$ is non-split. If $s_1$
(respectively $t_1$) is even and does not divide $q-1$, then it is
necessarily an even integer greater than $2$, thus it must divide
$q+1$, and so $B_1$ (respectively $C_1$) is non-split.
Lemma~\ref{lem.PGL.sigma} implies again that $\Sigma(B_1) \cap
\Sigma(A_2) \neq \{1\}$ (respectively $\Sigma(C_1) \cap \Sigma(A_2)
\neq \{1\}$), yielding a contradiction. Similarly, if $s_2$
(respectively $t_2$) is even, then it necessarily divides $q+1$.
Hence, condition~{\it(iii)} is necessary.

Moreover, if $C_1$ (respectively $C_2$) has order $2$, then the
above argument shows that it is necessarily split (respectively
non-split). By Corollary~\ref{cor.2.triple}, $(r_1,s_1,2)$
(respectively $(r_2,s_2,2)$) is a \emph{good involuting triple w.r.t
$q$} (see Definition~\ref{def.2.triple}), implying that
condition~{\it(iv)} is necessary.

\subsubsection*{The conditions are sufficient}
Let $(r_1,s_1,t_1)$ and $(r_2,s_2,t_2)$ be two hyperbolic triples of
integers. Assume that $G_1=\PGL_2(q)$ is a quotient of the triangle
groups $T_{r_1,s_1,t_1}$ and $T_{r_2,s_2,t_2}$ with torsion-free
kernel. Then one can find elements $A_1,B_1,C_1,A_2,B_2,C_2$ in
$G_1$ of orders $r_1,s_1,t_1,r_2,s_2,t_2$ respectively, such that
$A_1B_1C_1=1=A_2B_2C_2$ and $\langle A_1,B_1 \rangle = G_1 = \langle
A_2,B_2 \rangle$, and so conditions~{\it(i)} and~{\it(ii)} of
Definition~\ref{defn.beau} are fulfilled.

We may assume that $A_1,A_2,B_1,B_2 \in G_1 \setminus G$ and that
$C_1,C_2 \in G$. Hence $r_1,r_2,s_1,s_2$ are even. Moreover, by
Theorem~\ref{thm.tri.PGL2q}, $(r_1,s_1,t_1)$ and $(r_2,s_2,t_2)$ are
\emph{irregular w.r.t $q^2$}.

The condition that $\gcd(r_1,r_2) \leq 2$ now implies that one of
$r_1,r_2$ divides $q-1$ and the other divides $q+1$. We may assume
that $r_1 \mid q-1$ and $r_2 \mid q+1$, and so $A_1$ is split and
$A_2$ is non-split. Lemma~\ref{lem.PGL.sigma} now implies that
$\Sigma(A_1) \cap \Sigma(A_2) = \{1\}$.

If $s_1>2$, then the condition that $s_1 \mid q-1$ implies that
$B_1$ is split. If $s_1=2$ then $(r_1,2,t_1)$ is a \emph{good
involuting triple w.r.t $q$} and so by Corollary~\ref{cor.2.triple},
$B_1$ is split. Lemma~\ref{lem.PGL.sigma} implies again that
$\Sigma(B_1) \cap \Sigma(A_2) = \{1\}$.

Similarly, if $s_2>2$, then the condition that $s_2 \mid q+1$
implies that $B_2$ is non-split. If $s_2=2$ then $(r_2,2,t_2)$ is a
\emph{good involuting triple w.r.t $q$} and so by
Corollary~\ref{cor.2.triple}, $B_2$ is non-split.
Lemma~\ref{lem.PGL.sigma} implies again that $\Sigma(A_1) \cap
\Sigma(B_2) = \{1\}$ and $\Sigma(B_1) \cap \Sigma(B_2) = \{1\}$.

If $t_1>2$ is even, then the condition that $t_1 \mid q-1$ implies
that $C_1$ is split, If $t_1=2$ then $(r_1,s_1,2)$ is a \emph{good
involuting triple w.r.t $q$} and so by Corollary~\ref{cor.2.triple},
$C_1$ is split. Lemma~\ref{lem.PGL.sigma} implies again that
$\Sigma(C_1) \cap \Sigma(A_2) = \{1\}$ and $\Sigma(C_1) \cap
\Sigma(B_2) = \{1\}$. If $t_1$ is odd, then necessarily
$\gcd(t_1,r_2)=1$ and $\gcd(t_1,s_2)=1$, and
Lemma~\ref{lem.PGL.sigma} implies that $\Sigma(C_1) \cap \Sigma(A_2)
= \{1\}$ and $\Sigma(C_1) \cap \Sigma(B_2) = \{1\}$.

Similarly, if $t_2>2$ is even, then the condition that $t_2 \mid
q+1$ implies that $C_2$ is non-split, and if $t_2=2$ then
$(r_2,s_2,2)$ is a \emph{good involuting triple w.r.t $q$} and so by
Corollary~\ref{cor.2.triple}, $C_2$ is non-split.
Lemma~\ref{lem.PGL.sigma} implies again that $\Sigma(A_1) \cap
\Sigma(C_2) = \{1\}$ and $\Sigma(B_1) \cap \Sigma(C_2) = \{1\}$. If
$t_2$ is odd, then necessarily $\gcd(r_1,t_2)=1$ and
$\gcd(s_1,t_2)=1$, and Lemma~\ref{lem.PGL.sigma} implies that
$\Sigma(A_1) \cap \Sigma(C_2) = \{1\}$ and $\Sigma(B_1) \cap
\Sigma(C_2) = \{1\}$. Moreover, either $\gcd(t_1,t_2)=1$, or
$\gcd(t_1,t_2)=2$ and $C_1$ is split while $C_2$ is non-split, and
so, by Lemma~\ref{lem.PGL.sigma}, $\Sigma(C_1) \cap \Sigma(C_2) =
\{1\}$.

To conclude, $\Sigma(A_1,B_1,C_1) \cap \Sigma(A_2,B_2,C_2) = \{1\}$,
hence condition~{\it(iii)} of Definition~\ref{defn.beau} is
fulfilled.



\begin{thebibliography}{MM}


\bibitem{BCG05}
I. Bauer, F. Catanese, F. Grunewald, \textit{Beauville surfaces
without real structures}, Geometric methods in algebra and number
theory, Progr. Math., vol. \textbf{235}, Birkh\"auser Boston,
(2005), 1--42.

\bibitem{BCG06}
I. Bauer, F. Catanese, F. Grunewald, \textit{Chebycheff and Belyi
polynomials, dessins d'enfants, Beauville surfaces and group
theory}, Mediterr. J. Math. \textbf{3}, \textbf{no.2}, (2006),
121--146.

\bibitem{Be}
A. Beauville, \textit{Surfaces alg\'{e}briques complexes},
Ast\'{e}risque \textbf{54}, Paris (1978).

\bibitem{Magma}
W. Bosma, J. Cannon, C. Playoust, \emph{The Magma algebra system I:
The user language}, J. Symbolic Comput. \textbf{24} (1997), no.\
3--4, 235--265.

\bibitem{Cat00}
F. Catanese, \textit{Fibred surfaces, varieties isogenous to a
product and related moduli spaces}, Amer. J. Math. \textbf{122},
(2000), 1--44.





\bibitem{Di}
L.E. Dickson, \textit{Linear groups with an exposition of the Galois
field theory} (Teubner, 1901).


\bibitem{FMP}
B. Fairbairn, K. Magaard, C. Parker, \textit{Generation of finite
simple groups with an application to groups acting on Beauville
surfaces}, to appear in Proc. London Math. Soc.

\bibitem{FG}
Y. Fuertes, G. Gonz\'{a}lez-Diez, \textit{On Beauville structures on
the groups $S_n$ and $A_n$}, Math. Z. {\bf 264} (2010), 959--968.

\bibitem{FGJ}
Y. Fuertes, G. Gonz\'{a}lez-Diez, A. Jaikin-Zapirain, \textit{On
Beauville surfaces}, Groups Geom. Dyn. {\bf 5} (2011), 107--119.

\bibitem{FJ}
Y. Fuertes, G. Jones, \textit{Beauville surfaces and finite groups},
J. Algebra {\bf 340} (2011), 13--27.

\bibitem{GLL}
S. Garion, M. Larsen, A. Lubotzky, \textit{Beauville surfaces and
finite simple groups}, J. Reine Angew. Math. {\bf 666} (2012),
225--243.

\bibitem{GP}
S. Garion, M. Penegini, \textit{New Beauville surfaces and finite
simple groups}, to appear in Manuscripta Math.

\bibitem{Go}
D. Gorenstein, \emph{Finite groups}, Chelsea Publishing Co., New
York, 1980.

\bibitem{GM}
R. Guralnick, G. Malle, \textit{Simple groups admit Beauville
structures}, to appear in J. London Math. Soc.

\bibitem{LR1}
U. Langer, G. Rosenberger, \textit{Erzeugende endlicher projektiver
linearer Gruppen}, Results Math. {\bf 15} (1989), no. 1-2, 119--148.

\bibitem{LR2}
F. Levin, G. Rosenberger,
\textit{Generators of finite projective linear groups. II.},
Results Math. {\bf 17} (1990), no. 1-2, 120--127.

\bibitem{LS04}
M.W. Liebeck, A. Shalev, \textit{Fuchsian groups, coverings of
Riemann surfaces, subgroup growth, random quotients and random
walks}, J. Algebra {\bf 276} (2004), 552--601.




\bibitem{Mac}
A.M. Macbeath, \textit{Generators of the linear fractional groups},
Number Theory (Proc. Sympos. Pure Math., Vol. XII, Houston, Tex.,
1967), Amer. Math. Soc., Providence, R.I. (1969), 14--32.


\bibitem{Mar09}
C. Marion, \textit{Triangle groups and $\PSL_2(q)$}, J. Group Theory
\textbf{12} (2009), 689--708.





\bibitem{Su}
M. Suzuki, \textit{Group Theory I}, Springer-Verlag, Berlin, 1982.




\end{thebibliography}
\end{document}